\documentclass[11pt]{amsart}
\usepackage{amssymb, amsthm, amsmath, graphicx, fancyhdr, indentfirst, newlfont, latexsym, wasysym,  textcomp} %, enumitem}

\usepackage{amsmath, amssymb, amsthm,  epsfig}
\numberwithin{equation}{section}
\usepackage[usenames,dvipsnames]{color}

\newcommand{\lie}{\mathcal G}

\newcommand{\B}{\mathcal B}

\newcommand{\al}{\alpha}

\newcommand{\p}{\partial}  %followed by _

\newcommand{\e}{\epsilon}

\newcommand{\sg}{S}
\newcommand{\matriceB}{E}

\theoremstyle{remark}
\newtheorem*{rem*}{Remark}
\newtheorem{rem}{Remark}
\theoremstyle{plain}
\newtheorem{prop}{Proposition}
\theoremstyle{plain}
\newtheorem{teor}{Theorem}
\newtheorem*{teor*}{Theorem}
\theoremstyle{definition}
\newtheorem{lemma}{Lemma}
\newtheorem{cor}{Corollary}
\newtheorem{defin}{Definition}

\newcommand{\lap}[1]{\Delta #1}

\newcommand{\G}{\mathbb{G}}
\newcommand{\N}{\mathbb{N}}
\newcommand{\R}{\mathbb{R}}

\title[Heat kernels and mean curvature flow of graphs]{Sub-Riemannian heat kernels  \\ and mean curvature flow of graphs.}

\author{Luca Capogna}\address{Institute for Mathematics and its Applications, University of Minnesota, Minneapolis, MN 55455 \\ Department of Mathematical Sciences,
University of Arkansas, Fayetteville, AR 72701}\email{lcapogna@uark.edu}
\author{Giovanna Citti}\address{Dipartimento di Matematica, Piazza Porta S. Donato 5,
40126 Bologna, Italy}\email{citti@dm.unibo.it}
\author{Cosimo Senni Guidotti Magnani}\address{Dipartimento di Matematica, Piazza Porta S. Donato 5,
40126 Bologna, Italy}\email{senni@dm.unibo.it}
\keywords{Heat kernels in Lie groups, mean curvature flow, discrete time-step approximations, nonlinear semigroups\\
L. Capogna is  partially funded by NSF award  DMS 1101478. G. Citti and C. Senni Guidotti Magnani are partially funded by CG-DICE ISERLES European Project }
\begin{document}

\begin{abstract}
We introduce a sub-Riemannian analogue of the Bence-Merriman-Osher algorithm \cite{MerrimanBence Osher}
and show that it leads to weak solutions of the horizontal mean curvature flow of graphs over sub-Riemannian Carnot groups. The proof follows the nonlinear semi-group theory approach originally introduced by L. C. Evans \cite{EVANS} in the Euclidean setting and is based on new results on
the relation between sub-Riemannian heat flows of characteristic functions of subgraphs and the horizontal mean curvature of the corresponding graphs.\end{abstract}

\thanks{The authors would like to thank L. Ambrosio for useful discussions and acknowledge the hospitality of the Institute for Mathematics and its Applications (IMA) where part of the work was developed. We are also very grateful to the referee for the careful reading of our manuscript and for his/her observations.}
\maketitle

\section{Introduction}

The study of mean curvature flow in the sub-Riemannian setting is still at a very early stage, with
several key properties, such as existence, uniqueness  and regularity, still unknown. The notion itself  of motion by mean curvature is  understood only in special cases, as in the evolution of graphs over groups.

In this paper we give two contributions to this topic. First we establish in  the Carnot group setting a formula relating mean curvature of a graph and the heat flow of the characteristic function of the corresponding  sub-graph (see Lemma \ref{plane}). This formula seems to be new also in the Riemannian setting (see Corollary \ref{Riemannian-setting}). Then, using such formula we prove convergence of an analogue of the Bence-Merrriman-Osher \cite{MerrimanBence Osher}  
 algorithm in the Carnot group setting (see Theorem \ref{MAIN}), yielding a easily implementable time-step approximation of the sub-Riemannian mean curvature flow of graphs over Carnot groups. This convergence can in turn be used to define a notion of mean curvature flow also across characteristic points.

\subsection{ Carnot groups}
A Lie group $\G$ is called a  homogeneous stratified Lie group if its Lie algebra admits a stratification $$\lie= V^1\oplus \cdots \oplus V^r\text{  with }
[V^i,V^j]=V^{i+j}\text{  and  }[V^i, V^r]=0.$$ We will
let $n$ denote the topological dimension of $\G$.
Given a positive definite bilinear form
$g_0$ on $V^1$ we call the pair $(\G,g_0)$ a Carnot group and the corresponding
left invariant metric $g_0$ a sub-Riemannian metric.  For each $X\in V^i$ let $d(X)=i$ be
the degree of $X$ and let $Q=\sum_{i=1}^ri \; {\mathrm{dim}}(V_i)$ the homogeneous dimension of $\G$.
Let us consider an orthonormal  basis $\{X_1,....,X_n\}$  of $\lie $ fitted to the stratification, i.e. such that
\begin{align}\label{horizontalG}\text{the horizontal layer }H_0 := V^1 \text{ of } \G \text{ is generated by  }X_1, \dots, X_m,
\end{align}
while $\{X_i\}_{d(X_i)=k}$   generates  $V^k$.
These assumptions allow to define homogeneous dilations $\delta_\lambda$ and an homogeneous pseudo-norm $|\cdot |_0$ on $\lie$ (and through exponential coordinates) on $G$.  If $v {\, \in \lie}$ is expressed in term of the element of the basis as $v=\sum_i v_i X_i$, 
\begin{equation}\label{def-norm}
\delta_\lambda (v) := (\lambda ^{d(1)} v_1, \cdots, \lambda ^{d(n)} v_n), \quad |v|_0^{2r!}:=\sum_{i=1}^n |v_i|^{2r!/d(i)}.\end{equation}
In this setting we call horizontal gradient of a function $u: \G\rightarrow \R$ the vector
\begin{equation}\label{nabladefine}\nabla_0 u= (X_1u, \cdots X_mu),\end{equation} and we say that $u$ is in $C^1_X$ if its
horizontal gradient exists and it is continuous.
The graph of a $C^1_X$ function $u: \G\rightarrow \R$ can be considered a surface in $\G\times \R$. The product $\G\times \R$ can be  endowed with the structure of Lie group and  in particular  with  a Carnot group structure by setting $X_{n+1}=\partial_{n+1}$, and $d(n+1)=1$, \begin{equation}\label{horizontalGR}
\text{ the  horizontal layer } \tilde H_0=H_0 \times \R \;\text{ of }\; \G\times \R \text{ is generated by } (X_1, \cdots, X_{m}, X_{n+1}).\end{equation} The metric $g_0$ is extended to a metric $\tilde g_0$ in $\tilde H_0$ by requesting that $X_{n+1}$ is orthogonal to $\lie$ and has unit norm.
 We call horizontal normal of the graph of $u$
the projection on the horizontal plane  $\tilde H_0$ of the Euclidean normal and set
$$\nu_0 = \frac{(-\nabla_0 u, 1) }{\sqrt{1 + |\nabla_0 u|^2}}.$$
 Let us explicitly note that  such  graphs do not contain any characteristic points, i.e. points where the horizontal space is contained in the tangent space.
In the literature several equivalent definitions of the  horizontal mean curvature $h_0$ have been proposed at  non characteristic points of a $C^2$ surface  $M$. To quote a few: $h_0$ can be defined in terms of the first variation of the area functional
 \cite{dgn:minimal, pau:cmc-carnot, chmy:minimal, RR1, Sherbakova, montefalcone,cdpt:survey}, as horizontal divergence of the horizontal unit normal
 or as limit of the mean curvatures $h_\e$ of suitable Riemannian approximating {metrics} $g_\e$
 \cite{cdpt:survey}.
If the surface is not regular, the notion of curvature can be expressed in the viscosity sense (we refer to  \cite{bieske}, \cite{bieske2},
\cite{wang:aronsson}, \cite{wang:convex}, \cite{lms},
\cite{baloghrickly}, \cite{magnani:convex},
\cite{CC} for viscosity solutions of PDE in the sub-Riemannian setting).

The curvature $h_0$ of the graph of a function $u$ can be written as
\begin{equation}\label{curvature} h_0= \frac{-1}{\sqrt{1+|\nabla_0 u|^2}  }\sum_{i,j=1}^m\bigg(\delta_{ij}-\frac{ X_i u X_j u}{1+|\nabla_0 u|^2}\bigg) X_i X_j u\end{equation}

The horizontal mean curvature flow of a graph over a
Carnot group $\G$ is the flow $t \rightarrow \tilde M_t :=\{(x,u(x,t))|x\in \G\}\subset \G\times \R$ in
which each point in the evolving manifold moves along the horizontal normal,
 with speed given by the horizontal mean curvature. The evolution of a family of graphs of functions
 $t\to u(\cdot, t)$  is then characterized by the following equation
  \begin{equation}\label{mcfg}
\p_t u= -Au, \quad \text{ where } Au=-\sum_{i,j=1}^m\bigg(\delta_{ij}-\frac{ X_i u X_j u}{1+|\nabla_0 u|^2}\bigg) X_i X_j u,
\end{equation}
Here we will consider  viscosity solution of this equation:

\begin{defin}\label{weak1}
A continuous function $u:\G\times (0,\infty)\to \R$ is a weak sub-solution (resp.  super-solution) of equation  \eqref{mcfg}
if  for every $(x,t)\in \G\times \R$ and every smooth $\phi:\G\times(0,\infty)\to\R$ such that $u-\phi$ has a maximum (resp. a minimum) at $(x,t)$ then
$$\p_t \phi(x,t)\le \ (resp. \ \ge) \  \sum_{i,j=1}^m\bigg(\delta_{ij}-\frac{X_i \phi(x,t)X_j \phi(x,t)}{1+|\nabla_0 \phi (x,t)|^2}\bigg) X_i X_i \phi(x,t).
$$
 Solutions are  functions which are  simultaneously  super-solutions and a sub-solutions.

\end{defin}

Existence and uniqueness of viscosity solutions to this equation, attaining an assigned initial condition $f$ has been established in  \cite{CC} (see also the recent \cite{Manfredi} as well as  \cite{DirDragoniVonReness} for a probabilistic interpretation of the flow).
One of the goals  of this paper is to provide a discrete approximation of this motion, called diffusion driven motion by mean curvature.

\subsection{Diffusion driven motion by mean curvature}
In the Euclidean setting, the motion by mean curvature can be obtained through an algorithm {introduced by}  Merriman, Bence, Osher \cite{MerrimanBence Osher}, which relates the mean curvature flow and the heat flow.

The algorithm
is organized in two steps\footnote{See Definition \ref{definflow} for the precise definition}:
\begin{itemize}
\item{given  $\lambda\in \R$ and a function $f$, the characteristic function $\chi_{\sg_\lambda }$ of each of its $\lambda-$(sub)level {sets} $\sg_{\lambda} = \{\tilde x:  \ f(\tilde x) >\lambda   \}$ is diffused for a time $t$,  via the subriemannian heat flow, giving rise to a smooth function  $w_\lambda(\tilde x, t)$. }
\item{at time $t$  one defines a new  function  $H(t)f$ by requiring that its $\lambda-$level sets be   $\{\tilde x: w_\lambda(\tilde x, t)>1/2 \} $.}
\end{itemize}
This procedure leads to the definition of a two step algorithm, which given a function $f$ allows to define a new function $H(t)f$. The mean curvature flow $t\to u(\cdot, t)$ with initial data $f$, can be recovered applying iteratively this two step algorithm  for shorter and shorter time intervals (or equivalently through iteratively applying the operator $H$).

The convergence of the Euclidean version of this scheme
has been proved by Evans \cite{EVANS} and Barles, Georgelin \cite{BarlesGeorgelin}. In particular Evans
gave a proof based on nonlinear semi-group theory and on a pointwise study  of heat flow of sets in terms of curvature of their boundaries.
We also refer to  Ishii \cite{Ishii}, Ishii, Pires, Souganidis \cite{IshiiPiresSouganidis},  Leoni \cite{Leoni} and
Chambolle, Novaga \cite{ChambolleNovaga} for {extensions} of this algorithm to more general setting.
Let us mention that other convergent schemes for nonlinear parabolic equations  have been proposed in  \cite{CrandallLions}, \cite{Deckelnick}, \cite{DeckelnickDziuks1}, \cite{DeckelnickDziuks2}, \cite{Walkington} and   \cite{Elliot}.

\subsection{Diffusion driven motion by horizontal mean curvature in Carnot groups}

One of our main results  is an extension of the algorithm in \cite{MerrimanBence Osher} to the degenerate parabolic setting of  sub-Riemannian Lie groups.
A motivation for such extension comes from problems of visual perception and modeling of the visual cortex.
Geometric models for  the visual cortex as a  contact structure go back to work by Hoffmann \cite{Hoffmann}, Petitot and Tondut \cite{PetitotTondut} and Citti, Petitot and Sarti \cite{CSP}.
Later, in \cite{CittiSarti}, Citti and Sarti  endowed this contact structure with a sub-Riemannian metric,
and proposed a model of perceptual completion based {on} two cortical mechanisms applied in sequence. This two-steps process lead to a sub-Riemannian diffusion driven  motion by horizontal mean curvature in the cortical structure which is responsible for formation of subjective surfaces.
In the model they propose they study the evolution of a graph, and this is one of the main reasons why in this paper we focus on  diffusion driven motion of {graphs}.

The diffusion driven motion by curvature problem can be formulated also in the general setting of evolving  arbitrary surfaces in {Carnot groups (where it could have a wealth of applications in the study of minimal surfaces and isoperimetric problems). In this wider context, however, the notion of horizontal mean curvature  is not yet known, because the presence of characteristic points cannot be excluded.} Although in this paper we do not study this more general setting we do provide a relation between mean curvature and heat kernels which could be used for that purpose (see Proposition \ref{taylorchi}).

In Section 2 we will begin the proof of our main theorem, by stating a preliminary technical result, which allows to express sub-Riemannian quantities as limit of Riemannian ones. In particular we express the heat kernel as limit of the corresponding Riemannian ones. This instrument follows from a generalization of the celebrated lifting method of Rothschild and Stein \cite{RS}, and in the time independent case has been developed in \cite{CM}. The adaptation to the heat operator provides a powerful and general instrument for transfer results from the Riemannian to the sub-Riemannian setting. The proof of the technical result can be found in the frecent preprint \cite{CCM}.

In Section 3 we {introduce} the diffusion driven motion by mean curvature, and we restate the definition of {the} mean curvature operator in the setting of non linear semi-groups.
In particular we recognize that the curvature operator $A$ defined in \eqref{mcfg} is non dissipative (see definition \ref{diss} below).
According to a Theorem  of Crandall and Liggett \cite[Theorem I, page 266]{crandall-liggett}
this allows to give a weak definition of solution of \eqref{mcfg} in terms of a non linear semi-group:

\begin{teor*}\label{CraLiggteo}
If $A$ is non dissipative operator on a Banach space $\B$, then for all $f\in \B$  the limit
\begin{equation}\label{semigroup}
M(t)(f):= \lim_{j\to \infty, \\\lambda j\to t} (I+\lambda A)^{-j} f
\end{equation}
exists  locally uniformly in $t$.
\end{teor*}

The main result of section 3 is the proof {that this notion of weak semi-group solution} coincides with the viscosity one, given in Definition \ref{weak1}:

\begin{prop}\label{generatore}
For every continuous and periodic function $f$, the weak semi-group solution $M(t)(f)$ is the viscosity solution $u$ of motion by curvature, with initial datum $f$.
\end{prop}

Note that a non linear semi-group provides us with a discrete approximation of this motion, but the approximation is {non-linear},
while we were looking for a linear discrete approximation.

In Section 4 we give a definition of the evolution operator $H$ analogous to the one defined in \cite{EVANS}. However we are able to provide a substantial simplification, since we work only with  graphs (and not with more general level sets), and we evolve only the subgraph of the given initial datum. Calling $\chi_\sg$ the characteristic function of its subgraph, we evolve it with the sub-Riemannian heat flow in  $\G\times\R$:
\begin{equation}\label{subheat}\p_t w = \sum_{d(i) =1}X_i^2w \quad w(\tilde x, 0) = \chi_\sg(\tilde x).\end{equation}
The new function $H(t)f$ is defined requiring that {its sub-graph is $\{\tilde x : w(\tilde x, t)<1/2\}$}.

Note that the right-hand side of the heat equation \eqref{subheat} is  a sum of squares of $m+1$ vector fields in $\R^{n+1},$  and consequently it is only degenerate parabolic. However, in view of the work of H\"ormander \cite{Hormander}, Rothschild and Stein \cite{RS} (see also Jerison and Sanchez-Calle \cite{Jerison-Sanchez-Calle})  such operator is {hypoelliptic} and  admits a fundamental solution $\Gamma$ with good estimates (see Proposition \ref{prop: proprieta sol fond literature}
and \eqref{heat-1}). We conclude  Section 4 with Theorem \ref{Hlip}, which contains many properties of the operator  $H$ and in particular

\begin{prop}
For every continuous periodic function $f$ and $g$
\begin{equation}\label{Hcomp}||H(t)f-H(t)g||\le ||f-g||\text{ for all } t\ge 0.\end{equation}
\end{prop}

Section 5  contains the main step  and the geometric core of the proof: we will prove that the operator $H$ locally approximates motion by horizontal mean curvature. The evolution of the characteristic function of a given set $\sg$ is expressed through the heat kernel $\Gamma$ as
$$w(\tilde x,t) =\int \Gamma(\tilde y^{-1} \tilde x,t) \chi_\sg(\tilde y) d\tilde y.$$
Following Evans \cite{EVANS}, we need to perform a point-wise asymptotic expansion of $w$ for $t$ near $0$,
assuming that the boundary of $\sg$ is sufficiently  smooth.
In  the Euclidean case the expression of $\Gamma$ is  explicit, and the proof in \cite[Theorem 4.1]{EVANS} is  a direct computation. In contrast, in our sub-Riemannian setting the heat kernel does not have in general explicit expression and  in order to extend Evans' result we are obliged to use {a} deeper method, based on geometrical properties of the space.  Our approach uses in full strength the properties of the sub-Riemannian heat kernel and its Riemannian approximation. In this extension, the results stated in Section 2 play a crucial role. Proceeding in this fashion we then establish some new formula, emphasizing a link between  the heat kernel and  the mean curvature, which are of  independent interest.

These new formula are  closely related to the notion of heat content developed by De Giorgi \cite{DeGiorgi1}, \cite{DeGiorgi2} in the Euclidean case, and recently extended to the sub-Riemannian case in  \cite{BramantiMirandaPallara}. In these works the integral of the function $w$ is considered  under low regularity on the set $\sg$, and it is proved that its derivatives tend to the perimeter. {Higher terms expansions of the integral of $w$ in the Euclidean case have been established in \cite{BergGall}  where it has been proved that next term in the expansion depends on the mean curvature. }
%
%$$\int w(\tilde x, t) d\tilde x = |E| - \frac{2 \sqrt{t P(E)}}{\sqrt{\pi}} + \frac{t}{2}\int_{\partial E} h(\tilde x) d\sigma (\tilde x) + o(t)$$

Here we establish instead a pointwise estimate of the function $w$ involving the curvature of $\partial \sg$

\begin{prop}\label{taylorchi} If $\sg$ is the subgraph of a smooth function $f$ and $\tilde x\in \partial \sg$, one has
%Note different sign in curvature
 $$w(\tilde x, t) = \frac{1}{2} - \sqrt{t} h_0 (\tilde x)\int_{\Pi}\Gamma(\tilde z, 1) d\sigma_0(\tilde z) + O(t), \text{ as }t\to 0.$$
 where $\Pi$ is the  intrinsic tangent plane to $\sg$ at the point $\tilde x$, as defined in the statement of Lemma \ref{plane}.
\end{prop}

Note that this expression can be considered as a weak definition of horizontal mean curvature,
which does not rely on differentiability of the surface or on other geometric considerations. The same argument provided in the proof of Proposition \ref{taylorchi} yields an analogue result
for arbitrary hypersurfaces of a Riemannian manifold (see Corollary \ref{Riemannian-setting}).

\begin{rem} See the proof of   \cite[Theorem 3.3]{AMM} and  \cite[Theorem 4.1]{EVANS} for similar results 
in the Euclidean setting. In both papers  the explicit form
of the Euclidean heat kernel plays a key role. \end{rem}

From the latter we infer:

\begin{prop} For every continuous periodic function $f$   one has
$$t^{-1} (I-H(t) )(f) (\tilde x) \rightarrow  -   A f (\tilde x),$$
 uniformly as $t\to 0$.
\end{prop}

This approximation result allows to apply a general theorem of Brezis and Pazy \cite{brezis-pazy},
which ensures the following:

\begin{teor}\label{BrezisPazyteo}
If $A$ is an $m-$dissipative operator $\{H(t)\}_{t\ge 0}$ {satisfying }\eqref{Hcomp} and
\begin{equation}\label{hyp}
(I+\lambda A)^{-1} f= \lim_{t\to 0^+} \bigg( I+\lambda t^{-1} (I-H(t)) \bigg)^{-1} f,
\end{equation}
for every $f$ and $\lambda>0$, then for every $f\in \bar D(A)$ and $t\ge 0$ then
one has \begin{equation}\label{thesis}
M(t) f= \lim_{j\to \infty} H\bigg(\frac{t}{j}\bigg)^j f , \text{ uniformly for } t \text{ in compact sets }
\end{equation}
\end{teor}

On the other side, since we have already proved the relation  between the weak semi-group and the viscosity solutions of motion by curvature, we can finally deduce our main approximation result for the curvature flow:
\begin{teor}\label{MAIN}
If $f$ is a continuous and periodic function, and $u$ is
 the unique viscosity solution to \eqref{mcfg}, with initial datum $f$,
 then
$$u(\tilde x,t)= \lim_{j\to \infty} H\bigg(\frac{t}{j}\bigg)^j f , \text{ uniformly for } \tilde x \in \G\text{ and } t \text{ in compact sets }.$$
\end{teor}

\section{The sub-Riemannian structure as limit of its Riemannian approximation}\label{approximate}

 In this section we recall that the sub-Riemannian structure $(\G,g_0)$ in a Carnot group can be interpreted as a degenerate limit  (in the Gromov-Hausdorff sense)
of Riemannian spaces $(\G,g_\e)$ as $\e\to 0$, \cite{Gromov}.
In \cite{CM}, Citti and Manfredini studied the relation between
the sub-Riemannian Laplace operator and its Riemannian approximation.
In particular they established bounds, uniform in {the approximating parameter} $\epsilon$, for the fundamental solution of the
$g_\e-$Laplace operators. {Here we recall a recent extension of these bounds  to the degenerate parabolic case  by Manfredini and the first two named authors \cite{CCM}.

%We will work in the Carnot groups $\G$ and   $\tilde \G=\G\times \R$, whose Lie algebras we denote by  $\lie$ and by $\tilde \lie$. {On $\tilde \G$ one can define a semi-norm $|\tilde x|_{\tilde\G}=|(x,x_{n+1})|_{\tilde\G}:=|x|_{\G}+|x_{n+1}| $, a distance $\tilde d(\tilde x,\tilde y) = |\tilde y^{-1} \tilde x|$},  and dilations $\delta_\lambda(\tilde x)=(\delta_\lambda(x), \lambda x_{n+1})$ analogous to the same objects defined  $\G$.

We will work in the Carnot groups $\G$ and   $\tilde \G=\G\times \R$, whose Lie algebras we denote by  $\lie$ and by $\tilde \lie$. {On $\tilde \G$ one can define a semi-norm $|\tilde x|_0=|(x,x_{n+1})|_0:=|x|_0+|x_{n+1}| $  where $|x|_0$ is as in \eqref{def-norm}, a distance $\tilde d_0(\tilde x,\tilde y) = |\tilde y^{-1} \tilde x|_0$},  and dilations $\delta_\lambda(\tilde x)=(\delta_\lambda(x), \lambda x_{n+1})$ analogous to the same objects defined  $\G$.

For each $\e>0$ set $$X_{1}^\e= X_1, \cdots, X_{m}^\e = X_m, X_{m+1}^\e=\e X_{m+1}, \cdots  X_{n}^\e =\e X_{n}, \text{ and }X_{n+1}^\e=X_{n+1}$$
and  $\nabla_\e u = (X_{i}^\e)_{i=1, \cdots, n}$. We extend both
$g_0$  and $\tilde g_0$ to  Riemannian metrics $g_\e$  and $\tilde g_\e$ on $\lie$ and $\tilde \lie$ by requesting that the vectors $X_{i}^\e$ are an orthonormal family. The degree of these new vector fields is $d_\e(X_i^\e)=1$ for $\e>0$ and $i=1,...,n+1$.
The Carnot-Caratheodory distances (see \cite{NSW})  associated to these vector fields can be proved to be equivalent to the pseudo-distances
$d_\e(x,y) := |y^{-1} x|_{\e},$ and $\tilde d_\e(\tilde x,\tilde y) := |\tilde y^{-1} \tilde x|_{ \e},$
associated to the
 non-homogeneous norms. For every $v\in \lie$ and $v_{n+1}\in \R$,
 \begin{equation}\label{e-not-e}
 |v|_\e := \sum_{i=1}^m |v_i| + \sum_{i=m+1}^n \mathrm{min} \Big(\frac{|v_i|}{\epsilon}, |v_i|^{1/d(i)}\Big) \text{ and }
 |(v,v_{n+1})|_\e := |v|_\e+|v_{n+1}|.
\end{equation}

It is then clear that for every $\e>0$ one has
$d_1(x,y) \leq d_\e(x,y) \leq d_0(x,y).$ Analogous considerations hold for $\tilde d_0$ and $ \tilde d_\e $.
It can be shown (see \cite{Gromov} and references therein) that $(\tilde \G, \tilde d_\e)\to (\tilde \G, \tilde d_0)$ in the Gromov-Hausdorff topology.

\subsection{Riemannian mean curvature of a graph}

As it is well known, the mean curvature of the graph of a regular function $u_\e$ in the metric $g_\e$ can be expressed in the form:
$$
h_\e=\frac{1}{\sqrt{1+|\nabla_\e u_\e|^2}}\sum_{i,j=1}^n a_{ij}^\e (\nabla_\e u_\e) X_i^\e X_j^\e u_\e\text{ with }a_{ij}^\e(\xi)=\bigg(\delta_{ij}-\frac{\xi_i\xi_j}{1+|\xi|^2} \bigg)
$$
From this expression it immediately follows that $h_\e \rightarrow h_0$ as $\epsilon\to 0$.

Also note that classical results establish existence, uniqueness and regularity for  solutions to the  elliptic equation \begin{equation}\label{reg1}
u_\e-\lambda \sum_{i,j=1}^n a_{ij}^\e (\nabla_\e u_\e) X_i^\e X_j^\e u_\e = f
\end{equation}
for any regular domain of $\lie$, which can be identified with $\R^n$ (see \cite{GT}).

\subsection{Riemannian and sub-Riemannian heat kernels estimates}

For $\e\ge 0$, we  denote by {$\Delta_\e=\sum_{d_\e(i)=1} X_i^\e$} the Laplace operator  in $\tilde G$ defined in terms of the metric $\tilde g_\e$.
The subelliptic heat operator is defined in \eqref{subheat} as a sum of squares of vector fields
$$L = (\p_t- \Delta_0 ) u$$
while its Riemannian counterpart is
$$L_\e = (\p_t- \Delta_\e ) u = (\p_t- \sum_{d(i)=1} X_i^2 - \e^2\sum_{d(i)>1} X_i^2) u$$

Both  operators have  fundamental solutions, which we  call $\Gamma$ (or occasionally $\Gamma_0$) and $\Gamma_\e$ respectively. Note that the fundamental solution is defined on $(\G\times \R\times \R^+)^2$. However
 due to the Lie group invariance, we can fix the fundamental solution with pole in $(0,0)$, and the value in any other point will be obtained by translation invariance, i.e. for $\e\ge 0$
% $$ \Gamma_{\tilde x} ((\tilde y,t),  (\tilde x,\tau)) =\Gamma_{\tilde x} (({\tilde x}^{-1}\tilde y,t-\tau),(0,0)).$$
 \begin{equation}\label{notazione} \Gamma_\e((\tilde y,t),  (\tilde x,\tau)) =\Gamma_\e (({\tilde x}^{-1}\tilde y,t-\tau),(0,0)).\end{equation}

For simplicity we will drop the $(0,0)$ term in the argument of $\Gamma$ and consider the fundamental solution as a function defined on $\R^{n+1}\times \R^+$. Let us recall some well known properties of $\Gamma$,  (see references in \cite{Jerison-Sanchez-Calle})
\begin{prop}\label{prop: proprieta sol fond literature}(Subelliptic Gaussian estimates $\e=0$)
For all $\tilde x \, \in \, \G\times \R$  and $t>0$ we have
\begin{enumerate}
 \item\label{item: rescaling property} (Rescaling property)
 
\begin{align*}
 \Gamma(\tilde x,t) = \dfrac{1}{t^{(Q+1)/2}}\Gamma(\delta_t(\tilde x), 1)
\end{align*}
\item\label{item: asymptotic decay} There  exist constants $C_1, C_2$ depending on $\G$ and $g_0$, such that, for all $\tilde x \, \in \, \G\times \R$ and $t>0$, we have
\begin{multline}
 t^{-(Q+1)/2} \, e^{-C_1 \, \frac{{\tilde d_0}^2 (\tilde x,0)}{t}} \leq \Gamma(\tilde x,t) \leq t^{-(Q+1)/2} \, e^{-C_2 \, \frac{{\tilde d_0}^2 \mathcal{}(\tilde x,0)}{t}}
 \\ \text{ and }
 \\
  |X^I \p_t^\al \Gamma (\tilde x,t)| \le C_1
t^{-(Q+1)/2-\frac{d(I)}{2}-\al} e^{-C_2^{-1} \tilde d_0^2(\tilde x,0) /t}, \label{eq: asymptotic decay}
\end{multline}
where we have let $I=(i_1,...,i_l)\in \{1,...,m\}^l,$  $X^I=X_{i_1}...X_{i_l}$, $d(I)=\sum_{h=1}^l d(i_h)$
and $Q=\sum_{i=1}^r i \dim(V^i)$ is the homogeneous dimension of $\G$.
\item\label{item: normalization} For every $t>0$,
\begin{equation}\label{=1}
 \int_{\G\times \R} \Gamma(\tilde x,t) \, d \tilde x \equiv 1.
\end{equation}
\end{enumerate}
\end{prop}

This result implies in particular that
for every fixed $\e\ge 0$ Gaussian estimates of the heat kernel $\Gamma^\e$ hold
with constants $C_1,C_2$ depending on $\e$.
In this paper we will need to use of uniform estimates of the heat kernels
$\Gamma^\e$ in terms of the limit kernel $\Gamma$   (see \cite{CCM}):

\begin{prop}\label{convergence}
There exist constants $C_1,C_2>0$ depending on $\G, g_0$, such that for every $\e> 0$, $\tilde x\in \tilde \G$ and $t>0$ one has
\begin{equation}\label{heat-1}
C_1^{-1} |B_\e(0,\sqrt{t})|^{-1} \exp(-C_2 \tilde d_\e(\tilde x,0)^2/t) \le \Gamma^\e(\tilde x,t)\le C_1|B_\e(0,\sqrt{t})|^{-1} \exp(-C_2^{-1}  \tilde d_\e(\tilde x,0)^2/t).
\end{equation}
For any $l, \al\in \N$ and $l-$tuple $I=(i_1,...,i_l)\in \{1,...,n\}^l$, there exists $C_3=C_3(l,k,\G,g_0)>0$ such that for every $\e\ge 0$, $\tilde x\in \tilde \G$ and $t>0$ one has
\begin{equation}\  |(X^\e)^I \p_t^\al \Gamma^\e (\tilde x,t)| \le C_3
|B_\e(0,\sqrt{t})|^{-1} t^{-\frac{l}{2}-\al} \exp(-C_2^{-1} \tilde d_\e(\tilde x,0)^2/t),
\end{equation}
where we have let   $(X^\e)^I=X_{i_1}^\e...X_{i_l}^\e$.

%\begin{equation}\  |(X^\e)^I \p_t^\al \Gamma^\e (\tilde x,t)| \le C_3
%|B_\e(0,\sqrt{t})|^{-1} t^{-\frac{d_\e(I)}{2}-\al} \exp(-C_2^{-1} \tilde d_\e(\tilde x,0)^2/t),
%\end{equation}
%where we have let   $(X^\e)^I=X_{i_1}^\e...X_{i_l}^\e$ and $d_\e(I)=\sum_{h=1}^l d_\e(i_h)$.
%
Moreover, as $\e\to 0$ one has
$$X_\e^I \p_t^\al \Gamma^\e\to  X^I \p_t^\al \Gamma$$
uniformly on compact sets and  in a dominated way on all $\G$.
.
\end{prop}

The proof of this  result can be found  in \cite{CCM} and is directly inspired to the arguments introduced in \cite{CM},  where the time independent case is studied, and a general procedure  (reminiscent of the lifting technique in \cite{RS}) is introduced for lifting the operator to a higher dimensional space.

\
\section{Nonlinear semigroups and horizontal mean curvature flow of graphs}\label{nonlinear groups}
In this section we will restate the definition of mean curvature operator and viscosity solution of mean curvature flow in the setting of non linear semi-group theory.

\begin{defin}\label{diss} Let $\B$ be a Banach space, and A a $\B-$valued nonlinear operator with domain $D(A)\subset \B$. We say that $-A$ is  $m-$dissipative if
\begin{itemize}\item{$R(I+\lambda A)=\B$ for every $\lambda>0$,}
\item{its resolvent $J_\lambda=(I+\lambda A)^{-1}$ is
a single-valued contraction}
\end{itemize}
\end{defin}

\bigskip

The curvature flow will be defined for regular functions, periodic in the following sense:

\begin{defin}
Let $\G$ be a Carnot group with a set of generators $\{e_1,...,e_n\}$ closed under Lie brackets, and denote by $Q\subset \G$ a fundamental domain, i.e. an open bounded set such that $e_iQ\cap Q=\emptyset$ for all $i=1,...,n$ and $\G=\bigcup_{\alpha}
e_1^{\al_1}...e_n^{\al_n} \bar Q$, where the union is taken among all $n-$tuples of integers $\al=(\al_1,...,\al_n)$. A function $u:\G\to \R$ is periodic with respect to $Q$ and $e_1,...,e_n$ if $u(e_ix)=u(x)$ for all $i=1,...,n$.
\end{defin}

From now on $\B$ will denote the space of continuous periodic functions in a Carnot group $\G$, and  $(\B,|| \cdot ||)$ will be the Banach space obtained by endowing $\B$ with the sup norm $||\cdot||$.
\begin{defin}\label{def-gen}
We say that $u\in \B$ belongs to the domain of $A$ if there exists $f\in \B$ and $\lambda>0$ such that
$u$ is a weak solution (in the sense of  Definition \ref{weak2} below)  of
\begin{equation}\label{generator}
u-\lambda \sum_{i,j=1}^m a_{ij}(\nabla_0 u) X_i X_j u=f, \text{ where }a_{ij}(\nabla_0 u):=\bigg(\delta_{ij}-\frac{ X_i u X_j u}{1+|\nabla_0 u|^2}\bigg)
\end{equation}
in $\G$. In this case we will write $$(I+\lambda A) u =f.$$
\end{defin}

\begin{defin}\label{weak2}
A continuous function $u:\G\to\R$ is a weak sub-solution (resp. a super-solution) of
\begin{equation}\label{intro-generator}
u-\lambda \sum_{i,j=1}^m\bigg(\delta_{ij}-\frac{ X_i u X_j u}{1+|\nabla_0 u|^2}\bigg) X_i X_j u=f,
\end{equation}

in $\G$ if for every $x\in \G$ and smooth $\phi:\G\to\R$ such that $u-\phi$ has a maximum (resp. a minimum) at $x$
one must have $$u(x)-\lambda\sum_{i,j=1}^m\bigg(\delta_{ij}-\frac{X_i \phi (x)X_j \phi(x)}{1+|\nabla_0 \phi(x)|}\bigg) X_i X_j \phi(x)\le (resp. \ \ge ) \ f(x).$$ Solutions are  functions which are  simultaneously  super-solutions and a sub-solutions.
\end{defin}

%%%%%%%%%%%%%%%% m-dissipative

\begin{prop} \label{propA}The operator $-A$ introduced in Definition \ref{def-gen} is $m-$dissipative
\end{prop}
\proof

We first prove that $ ||u-v|| \le ||J_\lambda u-J_\lambda v||.$
For $\lambda>0$ let $f= J_\lambda u= u+\lambda A u$ and $w= J_\lambda v =v+\lambda Av$,   in the weak sense of Definition \ref{def-gen}. For $\mu>0$ set $u^\mu$ and $v_\mu$ the   sup-convolutions defined in \cite{wang:convex}: $$u^\mu(x)=\sup_{y\in \G} \Big( u(y)-\frac{1}{\mu} d^{2r!}_0(x,y) \Big) \quad v_\mu:=\inf_{y\in \G} \Big( v(y)+\frac{1}{\mu} d_0(x,y)^{2r!} \Big)$$  In view of Wang's results in \cite{wang:convex}
one has that $u^\mu$ is a super-solution of \eqref{generator} and is  semiconvex. For each $\alpha, \mu>0$ we can then apply Jensen's lemma (see \cite{MR920674} or \cite[Theorem 3.2]{usersguide})
 to the semi-convex function in $\G \times \G$,
$$\phi_{\mu, \alpha}(x,y):=u^\mu (x)-v_{\mu}(y)-\frac{\al}{2}|x-y|_E^2,$$ where $|\cdot|_E$ denotes the Euclidean norm, and obtain a sequence of points $p_j^{\mu,\alpha}=(x_j^{\mu,\alpha}, y_j^{\mu,\alpha})\to p^{\mu,\alpha}=(x^{\mu,\alpha},y^{\mu,\alpha})$ with $p^{\mu,\al}$ a maximum point for the function $\phi_{\mu,\al}$ which in turn is twice differentiable at $p_j^{\mu,\al}$ and satisfies
$$|\nabla_E \phi_{\mu,\al}(p_j^{\mu,\al})|=o(1)\text{ and }-\gamma I \le D^2_E \phi^{\mu,\al}(p_j^{\mu, \al} )\le o(1)I \text{ as }j\to \infty,$$
where $\nabla_E$ and $D^2_E $ denote respectively the Euclidean gradient and Hessian,
for some choice of $\gamma>0$. This implies that for every $i,h=1,...,m$ $$\lim_{j\to \infty} |X_iu^\mu(x_j^{\mu,\al})-X_iv_\mu (y_j^{\mu,\al})|=0\text{ and }\lim_{j\to \infty} X_i X_h u^\mu (x_j^{\mu,\al})-X_i X_h v_\mu (y_j^{\mu,\al})\le 0.$$ Since by definition $u$ and $v$ are solution of the {PDE}'s $f=   u+\lambda A u$ and $w=v+\lambda Av$ respectively, so that $u^\mu$ and $v_\mu$ satisfy
$$u^\mu-f\le\lambda \sum_{i,k=1}^m a_{ik}(\nabla_0 u^\mu)X_iX_k u^\mu \text{ at } x_j^{\mu,\al} \text{ and }v_\mu-w\ge\lambda  \sum_{i,k=1}^m a_{ik}(\nabla_0 v_\mu)X_iX_k v_\mu  \text{ at } y_j^{\mu,\al}.$$ Subtracting the first inequality from the second and using the following fact proved in \cite[Lemma 3.2]{usersguide}
$$\lim_{\al\to\infty, \mu\to 0} \al |x^{\mu,\al}-y^{\mu,\al}|=0\text{ and }\lim_{\al\to\infty, \mu\to 0} \sup_{\G\times \G} \phi_{\mu,\al}(x,y)=\sup_{x\in \G} (u(x)-v(x)),$$ one concludes that $ ||u-v|| \le ||f-w||,$ which is the first part in the definition of $m-$dissipative operator.

We now establish that $R(I+\lambda A)=\B$. Let $f\in \B\cap C^{\infty}(\G)$ and for each $\e>0$ consider weak solutions $u_\e$ of the approximating elliptic PDE \eqref{reg1}.
A simple variant of the argument above shows that for every $\e>0$ one has $||u_\e||\le ||f||$ and, in view of the fact that the vector fields $X_i's$ are left-invariant,  that $$\forall y\in \G \;\sup_{x\in \G}{|u_\e(yx)-u(y)|} \le \sup_{x\in \G}|f(yx)-f(y)|.$$ The latter proves that $\{u_\e\}_\e$ is equi-continuous and equi-bounded.   For a subsequence one has $u_{\e_k}\to u\in \B$ as $\e_k \to \infty$
and the argument in \cite[Theorem 5.2]{CC} shows that $u$ is a
 weak solution of \eqref{generator}. In view of the comparison principle we deduce that the range of  $I+\lambda A$ is closed, hence it is the whole of $\B$.

\endproof

The proof of the following proposition is  exactly as the proof of the first part of \cite[Theorem 2.4]{EVANS}.

\begin{prop}\label{propB} The domain of $A$ is dense in $\B$.
\end{prop}

Propositions \ref{propA} and \ref{propB} allow us to invoke the previously recalled  Generation Theorem of Crandall and Liggett (see (\ref{semigroup}) above), and ensure that $-A$ generates a non linear semi-group $M(t)(f)$.

\bigskip

{\bf Proof of Proposition \ref{generatore}}

Let $\phi:\G\times (0,\infty)\to \R$ be a smooth function such that $u-\phi$ has a maximum point at $(x_0,t_0)\in \G\times (0,\infty)$. We can always assume that the maximum is strict adding a  power of the gauge distance as in \cite{CC}. We set  $$
u^s(x,t)=\Big((I+\frac{1}{s}A)^{-k} f\Big)(x)\text{ for all integers }s, k\text{ and }t\in \Big[\frac{k}{s}, \frac{k+1}{s}\Big) $$ and note that for a fixed $t>0$, as $s\to \infty$ one has $k\to \infty$ and $k/s\to t$. By Crandall and Liggett's convergence result one has that $u^s\to u$ uniformly globally. As a consequence one can find a  sequence $(x_s,t_s)\in \G\times (0,\infty)$ such that: \begin{itemize}\item[(i)]{
$(x_s,t_s)\to (x_0,t_0)$;} \item[(ii)]{ For fixed  $t_s>0$ the function $(u^s-\phi)(\cdot, t_s)$ has a maximum at $x_s$;}
\item[(iii)]{ $ (u^s-\phi(x_s, t_s)+\frac{1}{s^2}\ge (u^s-\phi) (x,t)$ for all $(x,t)\in \G\times (0,\infty)$.}\end{itemize}
Following \cite[Theorem 2.5]{EVANS} for each $s$ we let $k_s$ be such that $t_s\in [\frac{k_s}{s},\frac{k_{s+1}}{s})$ and in particular $t_s-\frac{1}{s}\in [\frac{k_{s-1}}{s}, \frac{k_s}{s})$. This choice yields
$$u^s(x,t_s)= (I+\frac{1}{s}A)^{-1} u^s(\cdot,t_s-\frac{1}{s})(x),$$ in the weak sense
of Definition \ref{def-gen}.
By virtue of  (ii) and (iii) above one has
$$\frac{\phi(x_s,t_s)-\phi(x_s,t_s-\frac{1}{s})}{\frac{1}{s}} - \sum_{i,j=1}^m\bigg(\delta_{ij}-\frac{X_i \phi X_j \phi}{1+|\nabla_0 \phi|^2}\bigg) X_i X_i \phi|_{(x_s,t_s)} \le \frac{1}{s}.$$
As $s\to \infty$, the latter completes the proof that $u$ is a sub-solution. In a similar fashion one can prove that $u$ is also a super-solution and conclude the proof.
\endproof

\section{A discrete  time-step operator based on the heat flow }

\subsection{Flow of subgraphs of functions belonging to $\B$}

Let $(\B,||\cdot||)$ be the Banach space introduced in Definition \ref{def-gen}.
 If $f\in \B$, we call subgraph of $\B$ the  subset of $\G \times \R$ defined as
  \begin{equation}\label{subgraph}\sg=\{\tilde x=(x, x_{n+1}): x_{n+1}\leq f(x) \}\subset\G \times \R.\end{equation}
  We give here the definition of  heat flow of {a} subgraph  $\sg$ defined in \eqref{subgraph}, as an extension of the analogous notion given by Evans in the Euclidean case \cite{EVANS}.
For all $t > 0 $ denote $w(x, t)$  the solution of the Cauchy problem
$$
\left \{
\begin{array}{ccc}
\lap{w} & = & \partial_t w \\
w(., \, 0) & = & \chi_\sg
\end{array}
\right.
$$
and define
\begin{align}
\mathcal{H}(t)(\sg) = \{\tilde x \, \in \, \G \times \R : w(\tilde x, t) \geq \dfrac{1}{2}  \} \label{eq: heat flow}
\end{align}

In order to prove that this functional induces a flow on $\B$, we need to show that if the subgraph of a periodic, continuous function $f:\G\to \R$ is evolved
through the flow $\mathcal{H}(t)$ then the new set is also the subgraph of
a periodic continuous function.

\begin{lemma}\label{defflow} Let $f\in \B$. For any $t>0$ there exists a function  $g\in \B$ such that
$$\mathcal H(t) \{ x_{n+1}< f(x)\}=\{x_{n+1}< g(x)\}.$$
\end{lemma}
\proof
Let $(y,y_{n+1})\in \p \{\mathcal H(t) \{ x_{n+1}< f(x)\}$ so that
$$\frac{1}{2}=\int_{ \{ x_{n+1}< f(x)\} } \Gamma({\tilde x}^{-1} \tilde y,t)d \tilde x
=\int_{\G}  \int_{-\infty}^{f(x)}\Gamma ({x}^{-1}y, y_{n+1}-x_{n+1},t) dx_{n+1} dx.$$
If we denote
\begin{equation}\label{Ff}F_f(y,y_{n+1}):= \int_{\G} \int_{-\infty}^{f(x)}\Gamma({x}^{-1}y, y_{n+1}-x_{n+1},t)  dx dx_{n+1}.\end{equation}
differentiating along  the variable $y_{n+1}$ yields
$$\p_{y_{n+1} }F_f(y,y_{n+1})= \int_{\G} \Gamma ({x}^{-1}y, y_{n+1}- f(x), t) dx<0.$$

The implicit function theorem  then implies that in a small neighborhood of every point $z\in \G$, there exists a continuous function
$g:B(z,r)\subset \G\to \R$ such that $F_f(y,g(y))=1/2$ for all $y\in B(z,r)$. On the other hand, by the strict monotonicity of the function $F_f$ as a function of its last variable, we immediately deduce that the function $g$ is defined on the whole group $\G$ with real values.

\endproof

This lemma allows to give the following definition of flow of functions in $\B$:
\begin{defin}\label{definflow}
If $f\in \B$ we define the operator $$H(t):\B\rightarrow \B, \quad H(t)(f)=g,$$
where $g$ is the unique function satisfying
$$\mathcal H(t) \{ x_{n+1}< f(x)\}=\{x_{n+1}< g(x)\}.$$
\end{defin}

Thanks to the explicit expression of $H$ we can establish comparison principle properties of the flow $H$.
\begin{teor} \label{Hlip}  For each $t \geq 0$ the flow  $H(t): \B \to \B$ just defined has the following properties
 \begin{enumerate}
\item If $f  \leq g$ then $H(t)f  \leq H(t)g$
\item If $C$ is a real constant, then $H(t)(f+C) = H(t)f +C$
\item $H(t)$ is a contraction on $\B$
$$
||H(t)f - H(t)g || \leq || f-g||
$$
 \end{enumerate}
\proof
Let us show assertion (1). If $f,g\in \B$, and $f\geq g$ we can use Lemma \ref{defflow}, and represent the function  $F_g$ defined in (\ref{Ff}) as
$$ F_g(y,y_{n+1})= F_f(y,y_{n+1})  +  \int_\G \int_{f(x)}^{g(x)}\Gamma({x}^{-1}y, y_{n+1}-x_{n+1},t)  dx dx_{n+1}.$$
Since from the definition of $F_f$ and $H(t)f$ we have
$$F_f(x,H(t)f(x))=F_g(x,H(t)g(x)) = 1/2,$$
we immediately get
$$F_f(x,H(t)f(x))= F_g(x,H(t)g(x)) \geq F_f(x,H(t)g(x)).$$
By the strict monotonicity of $F_f$ in its last variable we deduce:
$$H(t)f\leq H(t)g.$$

In order to prove assertion (2) we note that
$$F_{f}(y,y_{n+1}) = \int_\G \int_{-\infty}^{f(x)}\Gamma({x}^{-1}y, y_{n+1}-x_{n+1},t)  dx dx_{n+1}=$$
with the change of variable $z_{n+1}=x_{n+1}+C$,
$$ = \int_\G \int_{-\infty}^{f(x)+C}\Gamma({x}^{-1}y, y_{n+1}+C-z_{n+1},t)  dx dx_{n+1}=F_{f+C}(y,y_{n+1}+C)$$
For $y_{n+1}=H(t)f(y)$ we get
$$F_{f+C}(y,H(t)f(y)+C)=1/2,$$
which implies by the uniqueness that
$H(t)(f)(y)+ C = H(t)(f+C)(y)$.

Finally we prove assertion (3). Let us choose $x$ such that
$$H(t)f(x) - H(t)g(x) > ||H(t)f - H(t)g|| - \epsilon$$
and call
$$H(t)f(x)=\mu\quad H(t)g(x) =\lambda.$$
By assertion (2) we have
$$H(t)f(x) - H(t)({g- \lambda + \mu - \epsilon})(x)>0$$
By assertion (1) this implies that there exists a point $y$ such that
$$f(y) - (g(y) -  \lambda + \mu - \epsilon)>0,$$
which by definition of   $\lambda$ and $\mu$ implies
$$ ||H(t)f - H(t)g|| < H(t)f(x) - H(t)g(x) + \epsilon < f(y) - g(y) + 2\epsilon \leq ||f - g|| + 2\epsilon $$
Assertion (3) follows.
\endproof
\end{teor}

%%%%%%%%%%%%%%%%%%%%%%%%%%%%%%%%%%%%
\section{Sub-Riemannian diffusion of smooth non-characteristic graphs}\label{core}

%%%%%%%%%%%%%%%%%%%%

In this section we prove the geometrical core of the paper, which is the equality of the normal {speed} of a graph evolving by mean curvature and the normal speed of the same graph evolving under the heat flow defined in \eqref{eq: heat flow}.
According to the notation introduced in section 2, for $f\in \B$ we will denote by $d\sigma_1$ the surface measure element on $\p \sg=\{(x,f(x))\in\G\times\R|x\in \G\}$
induced by the metric $g_\e$, with $\e=1$ and by  $d\sigma_0$ the horizontal perimeter measure $d\sigma_0=\frac{\sqrt{1+|\nabla_0 u|^2}}{\sqrt{1+|\nabla_1 u|^2}} d\sigma_1$. The outer horizontal unit normal $\nu_0$ to $\sg$ is expressed in the frame $X_1,...,X_{n+1}$ as $\nu_0=(-\nabla_0 f,1)/\sqrt{1+|\nabla_0 f|^2}$.
\begin{lemma}\label{intbyparts}
If $f:\G\to\R$ is a smooth function, with {sub-}graph $\sg$, then  for
 every  $i=1,...,m$ one has
$$\int_\sg  v(x) X_i u (x)dx= -\int_\sg u(x) X_i v(x) dx + \int_{\p \sg} u(x) v(x) \nu^0_i d\sigma_0(x),$$
for any smooth functions $u,v$ for which the integrals are defined.
\end{lemma}
\proof Integration by parts yields
$$\int_{\sg}  v(x) X_i u (x)dx= -\int_{\sg} u(x) X_i v(x) dx + \int_{\p \sg} u(x) v(x) \frac{X_i f}{\sqrt{1+|\nabla_1 u|^2}} d\sigma_1(x)$$
$$= -\int_{\sg} u(x) X_i v(x) dx + \int_{\p \sg} u(x) v(x) \frac{X_i f}{\sqrt{1+|\nabla_0 f|^2}}
d\sigma_0(x).$$
\endproof

\begin{lemma}\label{plane2}
Let  $\B$ be the space of periodic functions defined in Definition \ref{def-gen}. If $f\in \B\cap C^2(\G)$ let  $\sg$ its sub-graph (defined in  \eqref{subgraph}, and let $\tilde x=(x,x_{n+1})=(x,f(x))$ be a fixed point in $\p \sg$. Given $ L>0$  for every  $w \, \in \,  Lip (\G \times \R)$ with Lipschitz norm bounded by $L$,
 one has
$$\int_{ \p \sg}  \Gamma({\tilde x}^{-1}\tilde y, \tau)w(\tilde y) \, d \sigma_0 (\tilde y) =
\frac{w(\tilde x)}{\sqrt{\tau}}\int_{\Pi} \, \Gamma( z, 1) \, d\sigma_0(z) + O(1)$$
 as $\tau \rightarrow 0$ uniformly for $\tilde x \in \partial \sg$. Here
 $\Pi$ is the intrinsic tangent plane $y_{n+1}=\Pi( y):=f(x)+\sum_{d(i)=1} X_i f(x) (y_i-x_i)$ to ${\partial \sg}$ at $\tilde x$.
\end{lemma}
\begin{proof}
Since $\p \sg$ is a graph, the integration variables are of the form $\tilde y=(y,f(y))$.
Note that
$$({\tilde x}^{-1} \tilde y )=({x}^{-1} y, f(y)-f(x)) = ({x}^{-1} y, \sum_{d(i)=1}X_i f(x) (y-x)_i +
O(|{y}^{-1} x|^2)).$$
Consequently, setting $\tilde z=\delta_{\tau^{-1/2}} (\tilde x^{-1}\tilde y)$ one obtains
\begin{multline}
\tilde z = (z,z_{n+1})= (\delta_{\tau^{-1/2}} ({x}^{-1}y),\tau^{-1/2}( f(y)-f(x)))
\\=(\delta_{\tau^{-1/2}} ({x}^{-1}y), \sum_{d(i)=1} X_i f(x) (\delta_{\tau^{-1/2}} ({x}^{-1}y))_i + O(\sqrt{\tau} |\tilde z|^2)= (z,\Pi(z)+O(\sqrt{\tau} |z|^2) ).
\end{multline}
In view of the latter and using the scaling properties of the heat kernel one has
\begin{multline}
\Gamma({\tilde x}^{-1}\tilde y,\tau)=\tau^{-Q/2-1/2} \Gamma(\tilde z,1)= \tau^{-Q/2-1/2} \, \Gamma((z,
\Pi(z)+O(\sqrt{\tau} |z|^2),1)\\ =
\tau^{-Q/2-1/2} \bigg( \Gamma((z,\Pi(z)), 1) + R (z)\bigg)
\end{multline}
where $R$ is a remainder satisfying $\int_{\G} R dz= O(\sqrt{\tau}|z|^2).$
Consider the surface integral
\begin{multline}\label{surfaces} \int_{ \p \sg}  \Gamma({\tilde x}^{-1}\tilde y, \tau)w(\tilde y) \,  d \sigma_0(\tilde y) \\= w(\tilde x)
\int_{\p \sg} \Gamma({\tilde x}^{-1}\tilde y, \tau) \, d \sigma_0(\tilde y)
+ O\Big(
\int_{\G \times \R} |{\tilde x}^{-1}\tilde y| \Gamma({\tilde  x}^{-1}\tilde y, \tau) \,  d \sigma_0(\tilde y) \Big)\\
= w(\tilde x) \int_{\G} \Gamma ({  x}^{-1}  y, f(y)-x_{n+1}, \tau) \sqrt{1+|\nabla_0 f(y)|^2} \, dy \, +\\
 \quad \quad \quad \quad \quad  + O\Big(\int_{\G} |(x^{-1}y,f(y)-x_{n+1}) | \, \Gamma ({  x}^{-1}  y, f(y)-x_{n+1}, \tau) \sqrt{1+|\nabla_0 f(y)|^2} \, dy \Big)
\end{multline}

Applying the change of variable  $ z= \delta_{\sqrt{\tau}}({ x}^{-1} \, y)  $ we have $d y= \tau^{Q/2} d z$ and obtain

\begin{multline}
\int_{\G} \Gamma ({  x}^{-1}  y, f(y)-x_{n+1}, \tau) \sqrt{1+|\nabla_0 f(y)|^2} dy  \\= \frac{1}{\sqrt{\tau}}\int_{\G}  \bigg( \Gamma((z,\Pi(z)), 1) + R (z)\bigg)  \sqrt{1+|\nabla_0 f(x+\delta_{\sqrt{\tau}}(z))|^2} dz\\
= \frac{1}{\sqrt{\tau}}\int_{\G}  \bigg( \Gamma((z,\Pi(z)), 1) + R (z)\bigg)  \sqrt{1+|\nabla_0 \Pi|^2}dz \  + O(1)\\
= \frac{1}{\sqrt{\tau}} \int_\Pi \Gamma (\tilde z,1) d\sigma_0(\tilde z) + O(1).
\end{multline}
where in the last line we have used the approximation
$\nabla f(x+\delta_{\sqrt{\tau}}(z))=\nabla f(x)+ O (\sqrt{\tau} |z|)$  and  we have denoted $\tilde z = (z , \Pi(z))$. The proof follows immediately from the latter and from \eqref{surfaces}.
\end{proof}

We are now ready to establish  the key geometric identity needed in the proof of Proposition \ref{taylorchi} and eventually
to prove Theorem \ref{MAIN}.

\begin{rem}
The previous lemma indicates that it would be plausible to have an identity of the form
$$\int_{ \p \sg}  |\nabla \Gamma({\tilde x}^{-1}\tilde y, \tau)|w(\tilde y) \, d \sigma_0 (\tilde y) =
\frac{C w(\tilde x)}{\tau} + O(\frac{1}{\sqrt{\tau}} )$$
However in the next lemma we show that for integrands with a special structure one can improve on such an identity an obtain a better decay rate.
\end{rem}

\begin{lemma}\label{plane} Consider $f\in \B$ a $C^3$ function and suppose that the horizontal mean curvature $h_0$ of its graph is bounded by  a positive constant $L$. Let  $\sg$ denote its sub-graph, defined in \eqref{subgraph}, and consider $\tilde x=(x,x_{n+1})=(x,f(x))$  a fixed point in $\p \sg$. Then
 one has
\begin{equation}\label{one-here}
\int_{\p \sg} \langle  \nabla_0 \Gamma({\tilde x}^{-1}\tilde y, \tau), \nu_0(\tilde y) \rangle_{\tilde g_0} d\sigma_0(\tilde y)=-
\frac{h_0(\tilde x)}{2\sqrt{\tau}}\int_{\Pi}\Gamma( \tilde x^{-1}\tilde z, 1)d\sigma_0(\tilde z) + O(1)
\end{equation}
as $\tau \rightarrow 0$, and the convergence is uniform with respect to the variable $\tilde x\in \partial \sg$, but depends on the parameter $L$. Here
$\Pi$ is the intrinsic tangent plane $y_{n+1}=\Pi(z):=f(x)+\sum_{d(i)=1} X_i f(x) (y_i-x_i)$ to $\sg$ at the point $\tilde x$, $d\sigma_0$ denotes the horizontal perimeter measure of $\Pi$ associated to $X_1,...,X_m, \partial_{n+1}$.
\end{lemma}
\begin{rem}\label{plane-1}
Note that if $\tilde x\notin \p\sg$ then there exists $\al=\al(\sg,\tilde x)>0$ such that
\begin{equation}\label{one-here-1}
\int_{\p \sg} \langle  \nabla_0 \Gamma({\tilde x}^{-1}\tilde y, \tau), \nu_0(\tilde y) \rangle_{\tilde g_0} d\sigma_0(\tilde y)=O(e^{-\al/\tau}).
\end{equation}
In particular the left-hand side is uniformly bounded by $O(\tau^{-1/2})$ for any point $\tilde x\in \G\times \R$.
\end{rem}
\begin{proof}
The proof proceeds in three steps.
We will start considering the extremely simplified case in which the group $\G$ is the Euclidean space and the horizonal tangent plane coincides with the whole tangent plane, so that the sub-Riemannian metric reduces to a Riemannian one, and $m=n$. For reader convenience we will keep also in this case the same notations used in the rest of the paper. In the first step the metric will be constant, in the second step we consider an arbitrary Riemannian metric, while the subriemannian case will be studied in the third step as limit of the Riemannian ones.
\\
{\bf Step 1}
Let us first assume that  $\G=\R^n$ with  the Euclidean group structure and that we have a constant coefficient Riemannian metric $\tilde g_0$ defined on the whole tangent space of $\G\times\R$.
Up to a  change of variables we can assume that $\tilde x$ is the origin,
that $(X_1, \cdots X_{n}) = (\partial_{y_1}, \cdots, \partial_{y_n}) $ is an orthonormal frame for
$T_0\p\sg$ and $X_{n+1}$ coincides with the outer unit normal $\nu_0$ to $\sg$ at $0$.
In these coordinates the matrix associated to the metric $\tilde g_0$ is the identity,
the heat operator (defined in Section 2) simply reduces to the standard Euclidean one  $\p_t-X^2_i$.

There exists a neighborhood $U\subset \G\times \R$ of the origin where
$U\cap \p\sg$ can be represented as a graph of a smooth function over $U\cap T_0\p\sg=U\cap\{\tilde y_{n+1}=0\}$.
We will still denote by $f$ this function
and by $(y,f(y))$ its graph. The normal is expressed  $\nu_0(\tilde y) d\sigma_0(\tilde y)=(-\nabla_0 f(y), 1) dy$
and, by the choice of coordinates,  $\nabla_0 f(0)=0$.
It follows that
$$
\int_{\p \sg} \langle  \nabla_0 \Gamma({\tilde x}^{-1}\tilde y, \tau) ,\nu_0(\tilde y) \rangle_{\tilde g_0} d\sigma_0(\tilde y)=
\int_{\p\sg\cap U}\langle \nabla_0 \Gamma({\tilde x}^{-1}\tilde y, \tau) ,\nu_0(\tilde y)\rangle_{\tilde g_0}  d\sigma_0(\tilde y)+O(e^{-\al/t})$$
(where $\alpha>0$ is a suitable constant)
$$= -\int_{U\cap T_0\p\sg} \frac{\langle(y, f(y)), (-\nabla_0 f(y), 1)\rangle_I}{2t} \Gamma(y, f(y), t)  dy + O(e^{-\alpha/t})
$$
where  $\langle \cdot, \cdot\rangle_I$ denotes the Euclidean scalar product. Extending in a periodic fashion $f$ to all of $T_0\p\sg$ and identifying the latter with $\R^n$ one obtains
$$
=-
 \int_{\R^n} \frac{ f(y) - y \nabla_0 f(y)}{2t} \Gamma(y, f(y), t)  dy + O(e^{-\alpha/t})
$$

The integrand function can be expanded as follows:
$$f(y) - y \nabla_0 f(y) = -\frac{1}{2}y_i y_j \partial_{ij}f(0) + O(|y|^{3}) $$
Applying the change of variable $y= \sqrt{t}  z$,
and the scaling property of the heat kernel in $\R^{n+1}$, we arrive at
\begin{align*} &
\int_{\p \sg} \langle \nabla_0  \Gamma(\tilde y, \tau), \nu_0(\tilde y) \rangle_{\tilde g_0} d\sigma_0(\tilde y)  = \frac{\partial_{ij}f(0)}{2\sqrt{t}}\bigg( \int_{\R^n} \frac{z_i z_j}{2} \, \Gamma(z, 0, 1)  \, dz + O(\sqrt{t})\bigg)\\
& = \frac{\partial_{ii}^2f(0) }{2\sqrt{t}} \,  \int_{\Pi} \Gamma(z,0,1) d\sigma_0(z) + O(1) =-
\frac{h(\tilde x)}{2\sqrt{\tau}}\int_{\Pi}\Gamma( \tilde z, 1)d\sigma_0(z)  + O(1),
\end{align*}
here we have used the fact that in this system of coordinates $h_0(0)=- \partial_{ii}f(0)$. This concludes the proof in the constant metric case. We explicitly remark that in the setting of step 1 the hyperplane $\Pi$ coincides with the tangent plane to $\sg$ at $\tilde x$.

\bigskip

{\bf Step 2 }  Let us now assume that the metric  $\tilde g_0 =  {\tilde \matriceB}^T \, \tilde \matriceB$ is an arbitrary Riemannian metric on $\R^{n+1}$, invariant with respect to a Lie group  $\G$. We can define  an orthonormal frame in $\R^{n+1}$
by setting $X_{i} =    ({\tilde \matriceB} ^{-1})^T_{ij}\partial_j,$ with $i,j=1,...,n+1$.
The associated heat operator  now contains first order terms
$$L= \partial_t - \sum_{i=1}^{n+1} X_{i}^2=
\partial_t  - {\tilde g}^{kj} \partial_{jk}^2
- {\tilde \matriceB}^{-1}_{ji}\partial_j({\tilde \matriceB}^{-1})^T_{ik} \partial_{k}=\partial_t  - {\tilde g}^{kj} \partial_{jk}^2
- {\tilde a}_k\partial_{k}$$
where $\{{\tilde g}^{ij}\}$ denotes the inverse of $g$, and   $\tilde a_{k} = {\tilde \matriceB}^{-1}_{ji}\partial_j({\tilde \matriceB}^{-1})^{T}_{ik} .$
As in step 1 we can assume without loss of generality that  $\nabla_0 f(\tilde x) = 0$. Consequently  $- a_k(\tilde x)=0$ and the associated  constant coefficient operator, obtained evaluating the coefficients at the point $\tilde x$, reduces to:
\begin{align*}
L_{\tilde x}= \partial_t - \sum_{i=1}^{n+1} X_{\tilde x,i}^2=
\partial_t  - g^{kj}(\tilde x) \partial_{jk}^2,
\end{align*}
where we have denoted $X_{\tilde x,i}$ the vector fields with coefficients evaluated at the point ${\tilde x}$. We will denote by  $\nabla_{\tilde x} =(X_{\tilde x,1}, \cdots X_{\tilde x, n+1})$ the gradient along such frozen coefficients vector fields.
We will also denote  by $\Gamma$ and $\Gamma_{\tilde x}$ the heat kernels of $L$ and $L_{\tilde x}$.
Finally we denote by $h_0$ and $h_{0,\tilde x}$  the mean curvatures of the graph with respect to the metric $g_0$ and with respect to the   frozen metric $g(\tilde x)$. Note that at the point $\tilde x$ one has $h_{\tilde x}({\tilde x})= h({\tilde x}).$

We want to emphasize that in this step of the proof we are never going to use the invariance properties of  the heat operator  $L$ and its   fundamental solution $\Gamma$. Because of this, the group structure itself  becomes irrelevant and the fundamental solutions will be considered as a function defined on $(\R^{n+1}\times \R^{+})^2$, thus requiring the notation $\Gamma((\tilde y,t),  (\tilde x,\tau))$
(see also discussion at the beginning of Section 2 and \eqref{notazione}).
%In view of this observation we remark that
%the result in this step 2 still holds in any Riemannian setting. See Remark \ref{Riemanniancase} for the precise statement.

Using the  Levi's parametrix method (see for instance \cite{MIRANDA}, \cite{Friedman}), it is known that
one has
\begin{equation}\label{parametrix0}
\Gamma((\tilde y,t),  (\tilde x,\tau))= \Gamma_{\tilde x} ((\tilde y,t),  (\tilde x,\tau)) + O\Big(\frac{1}{ (t-\tau)^{n/2}} e^{- C_2\frac{|\tilde y-\tilde x |^2}{t-\tau}}\Big).\end{equation}More precisely \begin{equation}\label{parametrix1}\Gamma((\tilde y,t),  (\tilde x,\tau))= \Gamma_{\tilde x} ((\tilde y,t),  (\tilde x,\tau)) + J((\tilde y,t),  (\tilde x,\tau))\end{equation}
where
$$J((\tilde y,t),  (\tilde x,\tau))= \int_{\R^{n+1}\times[\tau,t]}\Gamma_{\tilde z} ((\tilde y,t), (\tilde z, s))Z_1((\tilde z, s), (\tilde x,\tau)) d\tilde z ds + O\Big(\frac{1}{ (t-\tau)^{(n-1)/2}} e^{- C_2\frac{|\tilde y-\tilde x |^2}{t-\tau}}\Big)$$
and $$Z_1((\tilde y, t), (\tilde x,\tau))=\sum_{ij}({\tilde g}^{ij}(\tilde y) - {\tilde g}^{ij}(\tilde x))\partial_{ij}\Gamma_{\tilde x} ((\tilde y,t), (\tilde x,\tau))+ \sum_{i}({\tilde a}^{i}(\tilde y) - {\tilde a}^{i}(\tilde x))\partial_{i}\Gamma_{\tilde x} ((\tilde y,t), (\tilde x,\tau)),$$
so that, for some constant $C=C(\G,g)>0$,
$$|Z_1((\tilde y, t), (\tilde x,\tau))|\leq C \frac{\Gamma((\tilde y, t), (\tilde x,\tau))}{\sqrt{t-\tau}}.$$

Clearly one also has \begin{equation}\label{parametrix}
\nabla_0\Gamma((\tilde y,t),  (\tilde x,\tau))= \nabla_{0}\Gamma_{\tilde x} ((\tilde y,t),  (\tilde x,\tau)) + \nabla_{0}J((\tilde y,t),  (\tilde x,\tau)).\end{equation}

Let us consider the integral
\begin{equation}\label{gradientJ}
\int_{\p \sg} \langle \nabla_0 J((\tilde y,t),  (\tilde x,\tau)),\nu_0(\tilde y) \rangle_{\tilde g_0} d\sigma_0(\tilde y)=\end{equation}

$$=
\int_{\p \sg} \langle \Bigg( \int_{\R^{n+1}\times[\tau,t]}\nabla \Gamma_{\tilde z} ((\tilde y,t), (\tilde z, s))Z_1((\tilde z, s) , (\tilde x,\tau)) d\tilde z ds\Bigg), \
\nu_0(\tilde y)  \rangle_{\tilde g_0} d\sigma_0(\tilde y)=$$
$$=
\int_{\R^{n+1}\times[\tau,t]} \Big(\int_{\p \sg}\langle  \nabla_0 \Gamma_{\tilde z} ((\tilde y,t), (\tilde z, s)),
\nu_0(\tilde y) \rangle_{\tilde g_0} d\sigma_0(\tilde y)\Big) Z_1((\tilde z, s), (\tilde x,\tau)) d\tilde z ds
.$$
In view of Remark \ref{plane-1} there exists $C=C(\G,g,\sg)>0$ such that

$$
\Bigg|\int_{\p \sg} \langle \nabla_0 J((\tilde y,t),  (\tilde x,\tau)),\nu_0(\tilde y) \rangle_{\tilde g_0} d\sigma_0(\tilde y)\Bigg|\le C \int_{\R^{n+1}\times[\tau,t]} \frac{|Z_1((\tilde z, s), (\tilde x,\tau))| }{\sqrt{t-s}}d\tilde z ds + O(e^{-\alpha/(t-\tau) })=$$
$$\le C\int_\tau^t \frac{1}{\sqrt{(t-s)(s-\tau))}}\int_{\R^{n+1}}  \Gamma_{\tilde z}((\tilde y,t), (\tilde z, s)) d\tilde z ds +O(e^{-\alpha/(t-\tau) })=$$$$=O\bigg(\int_\tau^t \frac{ds}{\sqrt{(t-s)(s-\tau)}}\bigg) + O(e^{-\alpha/(t-\tau) })=O \Big(\int_0^1 \frac{dr}{\sqrt{r(1-r)}}\Big), $$
where, in the last line, we have used the change of variable $r=(s-\tau)/(t-\tau)$.

 Consequently,
$$\int_{\p \sg}  \langle \nabla_0 \Gamma((\tilde y, t), (\tilde x, \tau)), \nu_0(\tilde y) \rangle_{\tilde g_0} d\sigma_0(\tilde y)=$$$$=
\int_{\p \sg} \langle  \nabla_0\Gamma_{\tilde x}((\tilde y, t), (\tilde x, \tau)), \nu_0(\tilde y) \rangle_{\tilde g_0} d\sigma(\tilde y)+$$$$+\int_{\p \sg} \langle \nabla_0 J((\tilde y, t), (\tilde x, \tau)) ,\nu_0(\tilde y) \rangle_{\tilde g_0} d\sigma_0(\tilde y)+
O\Bigg( \int_{\p \sg} \frac{1}{(t-\tau)^{n/2}} e^{- C_2\frac{|\tilde x - \tilde y|^2}{t-\tau}} d\sigma(\tilde y)\Bigg).$$
Denoting by $d\sigma_{g(\tilde x)}$ the surface measure corresponding to the frozen metric $g(\tilde x)$ and using the statement in  Step 1, one has
\begin{multline}\label{step2-1}
\int_{\p \sg}  \langle \nabla_0 \Gamma((\tilde y, t), (\tilde x, \tau)), \nu_0(\tilde y) \rangle_{\tilde g_0} d\sigma_0(\tilde y)=-\frac{h(\tilde x)}{2\sqrt{t-\tau}}\int_{\Pi}\Gamma_{\tilde x}((\tilde y, 1), (\tilde x, 0))
d\sigma_{g(\tilde x)} (\tilde y) + O(1) \\
\text{ (using Euclidean rescaling) } \qquad =-\frac{h(\tilde x)}{2}\int_{\Pi}\Gamma_{\tilde x}((\tilde y, t-\tau), (\tilde x, 0))d\sigma_{g(\tilde x)} (\tilde y) + O(1)=
\end{multline}(by estimate \eqref{parametrix0})
\begin{equation}\label{step2-2}-\frac{h(\tilde x)}{2}\int_{\Pi}\Gamma((\tilde y, t-\tau), (\tilde x, 0))d\sigma(\tilde y)+ O(1)
%=
%-\frac{h(\tilde x)}{2\sqrt{t-\tau}}\int_{\Pi}\Gamma( (\tilde z, 1), (\tilde x, 0))d\sigma(\tilde y) + O(1).
\end{equation}
We explicitly remark that also in the setting of step 2 the hyperplane $\Pi$ coincides with the tangent plane to $\sg$ at $\tilde x$.

{\bf Step 3} Let us now assume that $\Gamma$ is a sub-Riemannian heat kernel corresponding to a sub-Riemannian metric $g_0$ in the Carnot group $\G$.  From Section \ref{approximate} one has a sequence of left-invariant Riemannian metrics $g_\e$ in $\G$ such that $(\G,d_\e)\to (\G,d_0)$ in the Gromov-Hausdorff topology and corresponding sequence of Riemannian heat kernels
$\Gamma_\e$ satisfying the uniform estimates in Proposition \ref{convergence}. Applying step 2 at every level $\e>0$ and using the fact that $g_\e$ are left-invariant, one has the identities
\begin{equation}\label{one-here-step3}
\int_{\p \sg} \langle  \nabla_\e \Gamma_\e({\tilde x}^{-1}\tilde y, \tau), \nu_\e(\tilde y) \rangle_{\tilde g_\e} d\sigma_\e(\tilde y)=-
\frac{h_\e(\tilde x)}{2}\int_{T_{\tilde x}\p\sg}\Gamma_\e( \tilde x^{-1} \tilde y, \tau)d\sigma_\e(\tilde y) + O(1), \text{ as }\tau\to 0,
\end{equation}
where the bounds in $O(1)$ are uniform in $\e$. If we denote by $\Gamma_{\tilde x, \e}$ the heat kernel corresponding to the frozen Riemannian metric $g_\e(\tilde x)$ (no longer left-invariant with respect to $\G$) then \eqref{step2-1} and \eqref{step2-2} yield
\begin{equation}\label{one-here-step3-frozen}
\int_{\p \sg} \langle  \nabla_\e \Gamma_\e({\tilde x}^{-1}\tilde y, \tau), \nu_\e(\tilde y) \rangle_{\tilde g_\e} d\sigma_\e(\tilde y)=-
\frac{h_\e(\tilde x)}{2\sqrt{\tau}}\int_{T_{\tilde x}\p\sg}\Gamma_{\tilde x,\e}( ( \tilde z, 1), (\tilde x,0))d\sigma_{g_\e(\tilde x)} (\tilde z) + O(1), \text{ as }\tau\to 0,
\end{equation}
where the bounds in $O(1)$ are uniform in $\e$.

 In view of Proposition \ref{convergence} and the dominated convergence theorem it follows that the left-hand side integrals
$$\int_{\p \sg} \langle  \nabla_\e \Gamma_\e({\tilde x}^{-1}\tilde y, \tau), \nu_\e(\tilde y) \rangle_{\tilde g_\e} d\sigma_\e(\tilde y)$$
converge to the left-hand side of \eqref{one-here} as $\e\to 0$.

Next we turn our attention to the right-hand side of \eqref{one-here}.  First we recall that at every point $h_\e\to h_0$ as $\e\to 0$ and $d\sigma_\e\to d\sigma_0$ as measures.
%Since $\Gamma_\e$ are invariant with respect to Euclidean dilations, with the natural change of variable $\tilde z = \tilde y /t$  we get
%$$\frac{1}{\sqrt{\tau}}\int_{T_0\p\sg}\Gamma_\e(  \tilde z, 1)d\sigma_0(\tilde z)=
% \int_{T_0\p\sg}\Gamma_\e( \tilde y, t)d\sigma_0(\tilde y).
%$$
Using  the uniform bounds in Proposition \ref{convergence}  along with dominated convergence we deduce that  $$\lim_{\e\to 0} \int_{T_0\p\sg}\Gamma_\e( \tilde y, t)d\sigma_\e(\tilde y)=\int_{T_0\p\sg}\Gamma( \tilde y, t)d\sigma_0(\tilde y)
$$
$$ \text{ (applying Lemma \ref{plane2}) } \qquad =
\frac{1}{\sqrt{\tau}}\int_{\Pi}\Gamma( \tilde z, 1)d\sigma_0(\tilde z) + O(1),
$$
%as we see, applying the intrinsic dilations $\tilde y = \delta_t \tilde z$
 \end{proof}
 which completes the proof.

\begin{rem}\label{Riemanniancase}
The result in Lemma \ref{plane} seems to be new even in the Riemannian setting  (Step 2 in the previous proof). Since we find that such extension may be of independent interest we state it explicitly.

\begin{cor}\label{Riemannian-setting}
Consider  a Riemannian manifold $(N,g)$ and a smooth embedded hypersurface $M\subset N$ endowed with the induced metric.
Denote by  $\Gamma$ be the  heat kernel on $(N,g)$,  and for $\tilde x\in N$  denote by $\Gamma_{\tilde x}$ the heat kernel in $N$ corresponding to the frozen metric $g(\tilde x)$ viewed as a function defined on
$(T_{\tilde x} N\times \R^+)^2$.
%Assume that $G$ is such that $\Gamma_G$ satisfies the Gaussian estimates in \eqref{item: asymptotic decay}.
%DO WE NEED THIS?
 We also consider
$d\sigma_g$, the induced  volume element  on $M$, the unit  vector field $\nu$ normal to $M$, and by $h$ the mean curvature of $M$. For
 For every $\tilde x\in M$ and $t>s>0$  one has
\begin{equation}\label{one-here-2}
\int_{M} \langle \nabla_g\Gamma((\tilde x,s),( \tilde y,t)) ,\nu(\tilde y\rangle_G d\sigma_g(\tilde y)=-
\frac{h(\tilde x)}{2\sqrt{t-s}}\int_{T_{\tilde x}M}\Gamma_{\tilde x}((\tilde x,0),(z,1) )d\sigma_{\tilde x}(z) + O(1)
\end{equation}
as $t \rightarrow s$ uniformly for $\tilde x\in M$. Here we have denoted by $d\sigma_{\tilde x}$ the volume element on $T_{\tilde x} M$ induced by the metric $g(\tilde x)$ on $T_{\tilde x} N$.\end{cor}

 The proof follows closely the argument in Step 2 of the previous Lemma \ref{plane}. As we already noted, in that argument  we never used the group law structure, and Step 2 can be applied for any Riemannian metric and independently of the presence of a group structure. The result is local and we used the splitting of the space in $\G\times \R$ only to express the boundary of $S$ as a graph. Since this can be always done, under suitable regularity assumptions, the result in Step 2 holds in any Riemannian metric.

% For its independent interest, we state it here in this general setting, with classical used Riemannian notations.
%We will denote  $ g$ a  Riemannian metric, $<\cdot, \cdot>_g $ the Riemannian induced scalar product, $\nabla^g$ the Riemannian gradient, $\nu^g = G^{-1} \, \nu^I$  the Riemannian normal, and $h^g:=g^{ij}\p_{ij}^2 f$ is the mean curvature of the graph of $f$. Then Step 2 ensures that
%
%\begin{align}\label{riemann}
% \int_{\partial \sg} < \nabla_{\widetilde y}^g \Gamma((\widetilde y , t), (\widetilde x,0)) , \, \nu^g(\widetilde y) >_g \, d \sigma^g(\widetilde y) = \dfrac{h^g(\widetilde x)}{2 \sqrt{t}} \, \int_{\Pi} \Gamma(z,1) d_0^g|\Pi|(z) + O(1).
%\end{align}
\end{rem}
\bigskip

We can now prove the asymptotic  expansion  stated in the Proposition \ref{taylorchi} in the introduction

\bigskip

{\bf Proof of Proposition \ref{taylorchi}}
Let $\rho^{\e}$ be a smooth mollification of $\chi_\sg$.
For any interval $(t_1,t_2)\subset \R^+$, from the definition of heat kernel,
%and denoting by $\nu_0=(-\nabla_0 f,1)/\sqrt{1+|\nabla_0 f|^2}$ the outer pointing horizontal normal to $\sg$,  one immediately has
\begin{multline} 0= \lim_{R\to \infty} \int_{t_1}^{t_2}\int_{  \{  |{\tilde x}^{-1}\tilde y|\le R\}}(\partial_t-\mathcal L)\Gamma({\tilde x}^{-1}\tilde y, \tau)\rho^{\e}(\tilde y)d\tilde y
= \Big(\int_{\G\times \R}\Gamma({\tilde x}^{-1}\tilde y, \tau)\rho^{\e}(\tilde y)d\tilde y \Big)_{t_1}^{t_2} \\
-
\lim_{R\to \infty} \int_{t_1}^{t_2} \int_{ \{ |{\tilde x}^{-1}\tilde y|=R\}}\!\!\!\!\!\langle \nabla_0\Gamma({\tilde x}^{-1}\tilde y, \tau), \nu_0(\tilde y)\rangle_0 \ \rho^{\epsilon}(\tilde y)d\sigma_0(\tilde y) d\tau
\\
+\int_{t_1}^{t_2} \int_{\G\times \R} \langle \nabla_0 \Gamma({\tilde x}^{-1}\tilde y, \tau),\nabla_0 \rho^{\epsilon}(\tilde y)\rangle_0 d \tilde y d\tau,
\end{multline}
Letting $t_1,\epsilon\to  0$, one obtains
 \begin{multline}0= \int_{\G\times \R}\Gamma({\tilde x}^{-1}\tilde y, t_2)\chi_\sg(\tilde y)d\tilde y - \lim_{t_1 \rightarrow 0}\int_{\G\times \R}\Gamma({\tilde x}^{-1}\tilde y, t_1)\chi_\sg(\tilde y)d\tilde y   \\
 -
\lim_{R\to \infty} \int_0^{t_2} \int_{\{|{\tilde x}^{-1}\tilde y|=R\}}\!\!\!\!\!\! \langle \nabla_0\Gamma({\tilde x}^{-1}\tilde y, \tau), \nu_0(\tilde y)\rangle_0 \chi_\sg(\tilde y)(\tilde y)d\sigma_0(\tilde y) d\tau
\\+\int_0^{t_2} \int_{\p \sg}\langle \nabla_0 \Gamma({\tilde x}^{-1}\tilde y, \tau),\nu_0(\tilde y) \rangle_0 d\sigma(\tilde y) d\tau,
\end{multline}
Let us note that since $\widetilde x \in \p \sg$ then  $$\lim_{t_1\rightarrow 0}\int_{\G\times \R}\Gamma({\tilde x}^{-1}\tilde y, t_1)\chi_\sg(\tilde y)d\tilde y= \frac{1}{2}$$
Hence
\begin{align}\nonumber\int_{G\times \R}\Gamma({\tilde x}^{-1}\tilde y, t_2)\chi_\sg(\tilde y)d\tilde y &= \frac{1}{2}-
\lim_{R\to \infty} \int_0^{t_2} \int_{\{ \ |{\tilde x}^{-1}\tilde y|=R\}}\!\!\!\!\!\!\! \langle \nabla_0\Gamma({\tilde x}^{-1}\tilde y, \tau), \nu_0(\tilde y)\rangle_0 \ \chi_\sg(\tilde y)d\sigma_0(\tilde y) d\tau \\ \nonumber &
+ \int_0^{t_2}\int_{\p \sg}\langle \nabla_0 \Gamma({\tilde x}^{-1}\tilde y, \tau),\nu_0(\tilde y)\rangle_0 d\sigma(\tilde y) d\tau,
\end{align}

Next we show that
\begin{equation}\label{small}
\lim_{R\to \infty} \int_0^{t_2} \int_{\{ \ |x^{-1}y|=R\}}\!\!\!\!\!\!\! \langle \nabla_0\Gamma({\tilde x}^{-1}\tilde y, \tau), \nu_0(\tilde y)\rangle_0 \ \chi_\sg(\tilde y)d\sigma_0(\tilde y) d\tau =0.
\end{equation}
To see this we recall  that the $d\sigma_0$ perimeter of $\p B(0,R)$  in $\G\times \R$ is
given by $\int_{\p B(0,R)} d\sigma_0 = C_{\G} R^{Q}$. From this and from the heat kernel estimates
one has
$$ \int_{\{|{\tilde x}^{-1}\tilde y|=R\}}\!\!\!\!\!\!\! \langle \nabla_0\Gamma({\tilde x}^{-1}\tilde y, \tau), \nu_0(\tilde y)\rangle_0 \ \chi_\sg(\tilde y)d\sigma_0(\tilde y) d\tau
\le C_{\G} R^{-2} \int_0^{t_2} \bigg(\frac{R^2}{\tau}\bigg)^{(Q+2)/2} e^{-\frac{R^2}{c \tau}} d\tau,$$
which implies \eqref{small}.

%%%%%%%%%%%%%%%%%%

Applying Lemma \ref{plane} one concludes
\begin{align}
\int_{\G\times \R}\Gamma({\tilde x}^{-1}{\tilde y}, t_2)\chi_\sg({\tilde y})d{\tilde y}  &=\frac{1}{2}
-\int_0^{t_2} \bigg(\frac{h_0(\tilde x)}{2\sqrt{\tau}}\int_{\Pi}\Gamma(z, 1)d\sigma_0(z) + O(1) \bigg) d\tau
\\ &=\frac{1}{2} - h_0(\tilde x) \int_{\Pi}\Gamma(z, 1)d\sigma_0(z)  \int_0^{t_2} \frac{1}{2\sqrt{\tau}} d\tau
+O(t_2)
\\ &=\frac{1}{2}
- h_0 (\tilde x)  \sqrt{t_2} \int_{\Pi}\Gamma(z, 1)d\sigma_0(z) + O(t_2).
\end{align}
%\end{proof}

\begin{teor}
 Let  $f\in \B$ be a smooth function, and denote $\sg$ its subgraph, defined
 in \eqref{subgraph}. If $x\in \G$, $\tilde x = (x, f(x))$
$\al \, \in \, \R^{m+1} $  and $t>0$, denote
\begin{align*}
 q(t) = (x,f(x))\mathrm{exp}(\sum_{d(i)=1}\al_i(t) \, X_i )
\end{align*}
Then we have
\begin{align*}
 q(t) \, \in \, \partial (\mathcal{H}(t)\sg) \Longrightarrow  \sum_{d(i)=1}\al_i\nu_0^i = -   h_0(\tilde x) \, t  + O(t^{3/2}) \text{ as }t\to 0,
\end{align*}
where $h_0(\tilde x)$ is the horizontal mean curvature of $\p \sg$ in $\tilde x$ and $\nu_0=(\nu_0^1,...,\nu_0^{n+1})$.

\end{teor}

\proof
From the definition of the heat flow of sets one has

\begin{align}\label{conclusion}
 \dfrac{1}{2} &  = \int_\sg \Gamma(q(t)^{-1}\tilde y,t) \, dy \notag\\
\intertext{considering the Taylor expansion of the integrand with respect to $v$ we obtain}
  & = \int_\sg \Gamma({\tilde x}^{-1} \tilde y, t)d\tilde y -   \sum_{d(i)=1}\al_i\int_\sg X_i |_{\tilde x} \Gamma({\tilde x}^{-1}\tilde y, t)\, d\tilde y + o(t) \notag\\
   & = \int_\sg \Gamma({\tilde x}^{-1} \tilde y, t)d\tilde y -   \sum_{d(i)=1}\al_i\int_\sg X_i |_x\Gamma({\tilde x}^{-1}\tilde y, t)\, d\tilde y + o(t) \notag\\
 \intertext{moreover applying Lemma \ref{intbyparts} we get}
   & = \int_\sg \Gamma({\tilde x}^{-1} \tilde y, t)d\tilde y -   \sum_{d(i)=1}\al_i\int_{\p \sg }\Gamma({\tilde x}^{-1} \tilde y, t) \nu_i^0 (\tilde y) d\sigma_0(\tilde y) + o(t). \notag\\
 \intertext{and applying Lemma \ref{plane2}}
     & = \int_\sg \Gamma({\tilde x}^{-1} \tilde y, t)d\tilde y -  \sum_{d(i)=1}\frac{\al_i\nu^i_0  (\tilde x) }{\sqrt{t}}\int_{\Pi }\Gamma(\tilde z, 1)  d\sigma_0(\tilde z) + O(t) \notag\\
   & = \dfrac{1}{2} - \bigg(\sqrt{t} h_0(\tilde x)  +  \frac{ \langle\al, \nu_0 \rangle_{g_0}}{\sqrt{t}}\bigg) \int_{\Pi }\Gamma(\tilde z, 1)  d\sigma_0(\tilde z) +O(t)) \notag\\
   & = \dfrac{1}{2} - \frac{1}{\sqrt{t}}\int_{\Pi }\Gamma(\tilde z, 1)  d\sigma_0(\tilde z)\Big( th_0(\tilde x)  +  \langle\al, \nu_0 \rangle_{g_0}+  O(t\sqrt{t})\Big)
\end{align}

where we have applied Lemma \ref{taylorchi} in the first integral, and Lemma \ref{plane} in the second, concluding the proof.
\endproof

 Next we derive two important Corollaries from the previous theorem, which will be the main ingredients in the proof of Theorem \ref{MAIN}.

\begin{cor} Choosing $\al = t \beta \nu_0 $, with $\beta\in \R$ in the previous Theorem, we deduce that, if $ q(t) = (\tilde x)\mathrm{exp}(t\beta\nu_0)\in \, \partial (\mathcal{H}(t)\sg)$
 then\footnote{Note the difference in sign with \cite[Formula (60), Theorem 4.1]{EVANS}, which arises because of our choice of unit  normal}  $$\beta =  - h_0(\tilde x) + O(\sqrt{t}) \text{ as }t\to 0,$$
where $h_0(\tilde x)$ is the horizontal mean curvature of $\p \sg$ in $\tilde x$
\end{cor}
 \begin{prop}\label{4.2}
Let $f\in \B$ be a smooth function  and for $t>0$ denote by $H(t) f$ its flow defined in Definition \ref{definflow}.  For every $x\in \G$ one has
$$(H(t)f) (x)-f(x)=-t h_0(x,f(x)) \sqrt{1+|\nabla_0 f(x)|^2}+ o(t),$$
where the convergence $o(t)/t\to 0$ is uniform as $t\to 0$.
\end{prop}

\proof
Choosing $\al =   ((H(t) f)(x')-f(x'))X_{n+1}$
we have
$$(H(t)f) (x)-f(x)=  <\al, \nu_0>_{g_0} \sqrt{1+|\nabla_0 f(x)|^2}=  -t h_0(x,f(x)) \sqrt{1+|\nabla_0 f(x)|^2}+ O(t)$$
\endproof

 We can now conclude the proof of the main result of the paper, Theorem \ref{MAIN}.
As in \cite{EVANS} the key technical tool in the proof is the non-linear version of Chernoff's formula established by Brezis and Pazy (see Theorem \ref{BrezisPazyteo}) in the introduction.

{\bf Proof of Theorem \ref{MAIN}.} We only have to show \eqref{hyp} for $\lambda=1$. To this end we set for $t>0$ and $f\in \B$,
$$u^t:=\bigg(I+t^{-1}(I-H(t))\bigg)^{-1} f, \ \  \text{ and }\ \  A^tu:= \frac{u-H(t) u}{t}.$$
In view of Proposition \ref{Hlip}  the operator $-A^t$ is $m-$dissipative, thus implying that for all $y\in G$ and $t>0$,
$$\sup_{x\in \G}|u^t(yx)-u^t(x)|\le \sup_{x\in \G}|f(yx)-f(x)|$$ and consequently that  $\{u^t\}_{t\in (0,1]}$ is a family bounded and equi-continuous.

Let $\phi\in C^{\infty}(G)$ such that $u-\phi$ has a  positive maximum at $x_0\in G$. We can always assume that the maximum is strict, adding a suitable  power of the gauge distance, as for example in \cite{CC}.
Since $u^{t_k}\to u$ uniformly on compact sets then one can find a sequence of
points $x_k\to x_0$ as $k\to \infty$ such that $u^{t_k}-\phi$ has a positive maximum at $x_k$. In view of Proposition \ref{Hlip} one has
$$(H(t_k)u^{t_k})(x_k)- (H(t_k)\phi^{t_k})(x_k) \le u^{t_kk}(x_k)-\phi(x_k),\text{ or equivalently }  A^{t_k}\phi(x_k)\le A^{t_k} u^{t_k}(x_k).$$ Since $u^t +A^tu^t=f$
then
\begin{equation}\label{5.14}
u^{t_k}(x_k)+ \frac{\phi(x_k)-(H(t_k)\phi)(x_k)}{t_k}\le f(x_k).
\end{equation}
Invoking Proposition \ref{4.2} with $\phi$ in place of $f$, one obtains
% $x_k$ in place of $x_0$ and $\mu= (H(t_k)\phi)(x_k)$ one obtains
 %$$(H(t_k)\phi)(x_k)=\phi(x_k)+ t_k \frac{K}{2}\sum_{i,j=1}^m\bigg(\delta_{ij}-\frac{X_i \phi(x_k)X_j \phi(x_k)}{1+|\nabla_0 \phi(x_k)|^2} \bigg)X_i X_j \phi(x_k)+ o(t_k), \text{ as }k\to \infty.$$
 % Substitution of the latter   in \eqref{5.14} yields
  $$ u^{t_k}(x_k) - \frac{1}{2}\sum_{i,j=1}^m\bigg(\delta_{ij}-\frac{X_i \phi(x_k)X_j \phi(x_k)}{1+|\nabla_0 \phi(x_k)|^2} \bigg)X_i X_j \phi(x_k)+ o(1) \le f(x_k).$$
 Letting $k\to \infty$ we establish that $u$ is a weak sub solution of \eqref{intro-generator} with $\lambda=1$. In a similar fashion one can prove that $u$ is a weak super-solution, concluding the proof.

%\endproof


\begin{thebibliography}{100}

\bibitem{AMM} Angiuli, L.,  Massari, U., and Miranda, M. Jr.
Geometric properties of the heat content, Manuscripta Mathematica, 4 (2012),1-33.

\bibitem{baloghrickly} Balogh, Z. M., and Rickly, M. Regularity of convex functions on Heisenberg
groups. Ann. Sc. Norm. Super. Pisa Cl. Sci. (5) 2, 4 (2003), 847-868.



\bibitem{BarlesGeorgelin} Barles, G. and Georgelin, C. A simple proof of convergence for an approximation scheme for
computing motions by mean curvature. SIAM Journal on Numerical Analysis, 32(2):484–500,
1995.

\bibitem{BergGall}
van den Berg, M. and  Le Gall, J.-F. Mean curvature and the heat equation. Math. Z.,
215(3):437–464, 1994.

 \bibitem{bieske} Bieske, T. On 1-harmonic functions on the Heisenberg group. Comm. Partial
Differential Equations 3-4, 27 (2002), 727-761.

\bibitem{bieske2} Bieske, T. Comparison principle for parabolic equations in the Heisenberg
group. Electron. J. Differential Equations (2005), No. 95, 11 pp. (electronic).

\bibitem{BramantiMirandaPallara} Bramanti, M., Miranda, M., Pallara, D. Two characterization of BV functions on Carnot groups via the heat semigroup, Inter. Math. Res. Not. 2011

\bibitem{brezis-pazy} Br\'{e}zis H.  and Pazy, A.  Convergence and approximation of semigroups of nonlinear operators in Banach spaces, J. Functional Analysis 9 (1972), 63-74.

\bibitem{CC} Capogna, L. and Citti, G. Generalized mean curvature
flow in Carnot groups, Comm. Partial Differential
Equations 34 (2009), no. 7-9, 937-956.

\bibitem{CCM} Capogna, L, Citti, G. and Manfredini, M., Uniform Gaussian bounds for  sub elliptic heat kernels and an application to the total variation flow of graphs over Carnot groups, preprint 
Arxiv 1212.666 (2012).

\bibitem{cdpt:survey}
Capogna,  L., Danielli, D., Pauls, S. and Tyson, J. An introduction to the Heisenberg group and the sub-Riemannian
isoperimetric problem, Progress in Mathematics, vol. 259, Birkhauser Verlag, Basel, 2007.


\bibitem{ChambolleNovaga} Chambolle, A.  and Novaga, M.  Approximation of the anisotropic mean curvature flow, Preprint,
(2005).

\bibitem{chmy:minimal} Cheng, J.-H., Hwang, J.-F., Malchiodi, A. and Yang, P. Minimal surfaces in pseudohermitian geometry, Ann. Sc.
Norm. Super. Pisa Cl. Sci. (5) 4 (2005), no. 1, 129-177.



\bibitem{CM} Citti, G. and  Manfredini, M.  Uniform Estimates of the fundamental solution for a family of hypoelliptic operators, Potential Analysis, 2006, 25, pp. 147 - 164


\bibitem{CittiSarti} Citti, G. and Sarti, A., A cortical based model of perceptual completion in the Roto-Translation space,
Journal of Mathematical Imaging and Vision, 2006, vol. 24, pp. 307 - 326.


\bibitem{usersguide} Crandall, M.G., Ishii, H. and Lions, P.-L. User's guide to viscosity solutions of second order partial differential
equations, Bull. Amer. Math. Soc. (N.S.) 27 (1992), no. 1, 1-67.


\bibitem{crandall-liggett} Crandall M. G., and Liggett, T. Generation of semigroups of nonlinear transformations on general Banach spaces,
Amer. Jour. Math. 93 (1971), 265-298.

\bibitem{CrandallLions} Crandall, M.G., Lions,  P.-L. Convergent difference schemes for nonlinear parabolic equations
and mean curvature motion, Numer. Math., 75 (1996), pp. 17–41.


\bibitem{dgn:minimal} Danielli, D., Garofalo, N. and Nhieu D.M. Sub-Riemannian calculus on hypersurfaces in Carnot groups, Adv.
Math. 215 (2007), no. 1, 292-378


\bibitem{Deckelnick} Deckelnick, K. Error bounds for a difference scheme approximating viscosity solutions of mean curvature flow, Interfaces Free Bound., 2 (2000), pp. 117–142.

\bibitem{DeckelnickDziuks1} Deckelnick,  K.  and Dziuk, G. Discrete anisotropic curvature flow of graphs, M2AN Math. Model.
Numer. Anal., 33 (1999), pp. 1203–1222.

\bibitem{DeckelnickDziuks2} Deckelnick,  K.  and Dziuk, G. Error estimates for a semi-implicit fully discrete finite element scheme for the mean curvature
flow of graphs, Interfaces Free Bound., 2 (2000), pp. 341–359.

\bibitem{DeGiorgi1} De Giorgi, E. Su una teoria generale della misura (r - 1)-dimensionale in uno spazio ad r
dimensioni., Annali di Matematica Pura ed Applicata. Series IV 36 (1954): 191-213, and
also Ennio De Giorgi: Selected Papers, edited by Ambrosio, L., G. Dal Maso, M. Forti, M.
Miranda, and S. Spagnolo, 79–99. Springer, 2006. English translation, Ibid., 58–78.



\bibitem{DeGiorgi2} De Giorgi, E. Nuovi teoremi relativi alle misure (r - 1)-dimensionali in uno spazio a r
dimensioni. Ricerche di Matematica, 4 (1955): 95-113, and also Ennio De Giorgi: Selected
Papers, edited by Ambrosio, L., G. Dal Maso, M. Forti, M. Miranda and S. Spagnolo, 128–44.
Springer, 2006. English translation, Ibid., 111–27.

\bibitem{DirDragoniVonReness} Dirr, N., Dragoni, F. and von Renesse, M., Evolution by mean curvature flow in sub-Riemannian geometries. Communications on Pure and Applied Mathematics, 9 (2), (2010) pp. 307-326.


\bibitem{Elliot} Elliott, C. M. Approximation of curvature dependent interface motion, in The state of the art in
numerical analysis (York, 1996), vol. 63 of Inst. Math. Appl. Conf. Ser. New Ser., Oxford Univ. Press,
New York, 1997, pp. 407–440.

%\bibitem{EsedogluTsai} Esedoglu, S. and Tsai, Y.H.R. Threshold dynamics for the piecewise constant Mumford-Shah
%functional. J. Comput. Phys., 211(1), 367–384, 2006.

\bibitem{EVANS} Evans, L. C. Convergence of an algorithm for mean curvature motion., Indiana Univ. Math. J. 42 (1993), no. 2,
533-557.

\bibitem{Manfredi} Ferrari, F.,   Manfredi, J.  and  Liu, Q., On the horizontal Mean Curvature Flow for Axisymmetric surfaces in the Heisenberg Group, preprint (2012).

\bibitem{Friedman} Friedman, A. Partial differential equations of parabolic type, Prentice-Hall, 1964. - XIV, 347 p.

\bibitem{GT} Gilbarg, D. and Trudinger, N., Elliptic Partial Differential Equations of Second Order,
Grundlehren der mathematischen Wissenschaften, 224, Springer-Verlag, Berlin, Heidelberg, New York, 1983.


\bibitem{Gromov} Gromov, M. Metric Structures for Riemannian and Non-Riemannian Spaces,
Progress in Mathematics, 152, Birkh\"auser, 1999.

\bibitem{pau:cmc-carnot} Hladky, R. K. and Pauls, S. D. ,Constant mean curvature surfaces in sub-Riemannian geometry,  J. Diff. Geom. 79 (2008) no.1, 111-139.


\bibitem{Hormander} H\"ormander, L., Hypoelliptic second order differential equations, Acta. Math.
119 (1967), 141-171.


\bibitem{Hoffmann} Hoffman, W. C.,  The visual cortex is a contact bundle, Applied Math. And Computation , 1989, 32
, 137-167.

\bibitem{Ishii} Ishii, H.  A generalization of the Bence, Merriman and Osher algorithm for motion by mean
curvature. In Damlamian, A. , Spruck, J. and Visintin, A. editors, Curvature Flows and Related
Topics, pages 111–127. Gakkˆotosho, Tokyo, 1995.

\bibitem{IshiiPiresSouganidis} Ishii, H., Pires, G.E. and Souganidis, P.E. Threshold dynamics type schemes for propagating fronts. TMU Mathematics Preprint Series, 4, 1996.

\bibitem{Jerison-Sanchez-Calle}    Jerison, D.S., Sánchez-Calle, A., Estimates for the heat kernel for a sum of squares of vector Fields, Indiana Univ. Math. J. 35 (4) (1986)
835–854.


\bibitem{MR920674} Jensen, R. The maximum principle for viscosity solutions of fully nonlinear second order partial differential
equations, Arch. Rational Mech. Anal. 101 (1988), no. 1, 1-27.

\bibitem{Leoni} Leoni, F.
Convergence of an Approximation Scheme for Curvature-Dependent Motions of Sets
Journal
SIAM Journal on Numerical Analysis archive
Volume 39 Issue 4, 2001
Pages 1115 - 1131

\bibitem{lms} Lu, G., Manfredi, J. J., and Stroffolini, B. Convex functions on the
Heisenberg group. Calc. Var. Partial Differential Equations 19, 1 (2004), 1-22.

\bibitem{magnani:convex} Magnani, V. Lipschitz continuity, Aleksandrov theorem, and characterizations
for H-convex functions. Math. Ann. 334 (2006), 199-233.


\bibitem{MerrimanBence Osher}  Merriman, B.; Bence, J.; Osher, S. J. Diffusion generated motion by mean cur-
vature. Proceedings of the Computational Crystal Growers Workshop, pp. 73-83.
Editor: Jean Taylor. AMS, Providence, Rhode Island, 1992.


\bibitem{MIRANDA} Miranda, C. Partial Differential Equations of Elliptic Type, Springer-Verlag, 1970.


\bibitem{montefalcone} Montefalcone, F.  Hypersurfaces and variational formulas in sub-Riemannian Carnot groups, J. Math. Pures
Appl. (9) 87 (2007), no. 5, 453-494.

\bibitem{NSW} Nagel, A., Stein, E.M. and Wainger, S. Balls and metrics defined by vector fields I: Basic properties, Acta
Math. 155, (1985), 103-147.

\bibitem{PetitotTondut} Petitot, J., Tondut, Y., 1999: "Vers une neurogéométrie. Fibrations corticales, structures de
contact et contours subjectifs modaux", Mathématiques, Informatique et Sciences Humaines, 145,
5-101.


\bibitem{CSP} Sarti A., Citti G., Petitot J. ,
The symplectic structure of the primary visual cortex 
(2008) Biological Cybernetics, 98 (1) , pp. 33-48.  

\bibitem{RR1} Ritor\'{e} M., and Rosales, C. Area stationary surfaces in the Heisenberg group H1, Adv. Math. 219 no. 2 (2008) 633-671.

\bibitem{RS}
L. P. Rothschild and E. M. Stein, Hypoelliptic differential operators and nilpotent groups, Acta Math. 137 (1976), 247-320.

\bibitem{Sherbakova}  Sherbakova, N.
 Minimal surfaces in contact sub-Riemannian manifolds and structure of their singular sets in the $(2,3)$ case, ESAIM: COCV 15 (2009) 839-862.

\bibitem{Walkington} Walkington, N. J. Algorithms for computing motion by mean curvature, SIAM J. Numer. Anal., 33
(1996), pp. 2215–2238

\bibitem{wang:aronsson}Wang, C.Y. The Aronsson equation for absolute minimizers of $L^\infty$ functionals associated with vector fields satisfying H\"ormander's condition.  Trans. Amer. Math. Soc. 359 (2007), 91-113

\bibitem{wang:convex}  Wang, C. Y.  Viscosity convex functions on Carnot groups, Proc. Amer. Math. Soc. 133 (2005), no. 4, 1247-1253.
(electronic).


\end{thebibliography}
\end{document}